\documentclass[a4paper,11pt]{article}
\usepackage{latexsym}
\usepackage{amssymb}
\usepackage{enumerate}
\usepackage{commath}
\usepackage{amsfonts}
\usepackage{amsmath,amsthm}
\usepackage{hyperref}
\usepackage{graphicx} 

\usepackage[paper=a4paper,left=30mm,right=20mm,top=25mm,bottom=30mm]{geometry}

\newenvironment{@abssec}[1]{%
    \if@twocolumn

      \section*{#1}%
    \else

      \vspace{.05in}\footnotesize
      \parindent .2in
 {\upshape\bfseries #1. }\ignorespaces
    \fi}

    {\if@twocolumn\else\par\vspace{.1in}\fi}

\newenvironment{keywords}{\begin{@abssec}{\keywordsname}}{\end{@abssec}}

\newenvironment{AMS}{\begin{@abssec}{\AMSname}}{\end{@abssec}}

\newcommand\keywordsname{Key words}
\newcommand\AMSname{AMS subject classifications}
\newcommand\AMname{AMS subject classification}
\newcommand\restr[2]{{
\left.\kern-\nulldelimiterspace 
#1 
\vphantom{|} 
\right|_{#2} 
}}
\newtheorem{theorem}{Theorem}[section]
 \newtheorem{lemma}[theorem]{Lemma}
 
 \newtheorem{proposition}[theorem]{Proposition}

\newcommand{\RR}{\mathbb{R}}

\renewcommand{\SS}{\mathbb{S}}

\def\XXint#1#2#3{{\setbox0=\hbox{$#1{#2#3}{\int}$}
\vcenter{\hbox{$#2#3$}}\kern-.5\wd0}}

\newcommand{\dist}{\mathop{\mathrm{dist}}}  
\newcommand{\link}{\mathop{\circ\kern-.35em -}}

\newcommand{\ol}{\overline}
\newcommand{\pa}{\partial}

\newcommand{\dv}{\mathop{\mathrm{div}}}

\newcommand{\gr}{\nabla}

\newcommand{\al}{\alpha}
\newcommand{\be}{\beta}
\newcommand{\ga}{\gamma}  
\newcommand{\Ga}{\Gamma}
\newcommand{\de}{\delta}
\newcommand{\De}{\Delta}
\newcommand{\ve}{\varepsilon}
 
\newcommand{\la}{\lambda}
\newcommand{\La}{\Lambda}    
\newcommand{\ka}{\kappa}
\newcommand{\si}{\sigma}

\newcommand{\Om}{\Omega}
\newcommand{\rn}{{\mathbb{R}}^N}
\newcommand{\sg}{\sigma}

\newcommand{\cF}{\mathcal{F}}
\newcommand{\cG}{{\mathcal G}}
\newcommand{\cH}{{\mathcal H}}

\newcommand{\cX}{\mathcal{X}}

\title{Two-phase heat conductors with a surface \\
of the constant flow property}

\author{Lorenzo Cavallina\thanks{Research Center for Pure and Applied Mathematics,
Graduate School of Information Sciences, Tohoku
University, Sendai, 980-8579, Japan ({\tt cava@ims.is.tohoku.ac.jp}, {\tt  sigersak@tohoku.ac.jp}).}\ ,
Rolando Magnanini\thanks{Dipartimento di Matematica U.~Dini, Universit\`a di Firenze, viale Morgagni 67/A, 50134 Firenze, Italy
({\tt rolando.magnanini@unifi.it}).}\ , and Shigeru Sakaguchi\footnotemark[1]}

\date{}

\begin{document}

\maketitle

\begin{abstract}
We consider a two-phase heat conductor in $\mathbb R^N$ with $N \geq 2$ consisting of a core and a shell with different constant conductivities. We study the role played by radial symmetry for overdetermined problems of elliptic and parabolic type.
\par
First of all, with the aid of the implicit function theorem, we give a counterexample to radial symmetry for  some two-phase elliptic overdetermined  boundary value problems of Serrin-type.
\par
Afterwards, we consider the following setting for a two-phase parabolic overdetermined problem. We suppose that, initially, the conductor has temperature 0 and, at all times, its boundary is kept at temperature 1. A hypersurface in the domain has the constant flow property if at every of its points the heat flux across surface only depends on time. It is shown that the structure of the conductor must be spherical, if  either there is a surface of the constant flow property in the shell near the boundary or a connected component of the boundary of the heat conductor is a surface of the constant flow property. 
Also, by assuming that the medium outside the conductor has a possibly different conductivity,  we consider a Cauchy problem in which the conductor has initial inside temperature $0$ and  outside  temperature $1$. We then show that a quite similar symmetry result holds true. 
\end{abstract}



\pagestyle{plain}
\thispagestyle{plain}

\section{Introduction}
\label{introduction}

\vskip 2ex
In this paper we examine several overdetermined elliptic and parabolic problems involving a two-phase heat conductor in $\rn$, which consists of a core  and a shell  with different constant conductivities. 

The study of overdetermined elliptic problems dates back to the seminal work of Serrin \cite{Se1971}, where he dealt with the so called torsion function, i.e. the solution to the following elliptic boundary value problem.
\begin{equation*}
-\Delta u =1\;\text{in }\Om, \quad u=0\;\text{ on }\pa\Om.
\end{equation*}
Serrin showed that the normal derivative of the torsion function $u$ is a constant function on the boundary $\pa\Om$ if and only if the domain $\Om$ is a ball. We remark that such  overdetermined conditions arise naturally in the context of critical shapes of shape functionals. In particular, if we define the torsional rigidity functional as $T(\Omega)=\int_\Om u\,dx$, then Serrin's overdetermination on the normal gradient of $u$ is equivalent to the shape derivative of $T$ vanishing for all volume preserving perturbations (we refer the interested reader to  \cite[chapter 5]{henrot}). 

As far as overdetermined parabolic problems are concerned, we refer for example to \cite{AG1989ARMA}, where symmetry results analogous to Serrin's one are proved as a consequence of an overdetermination on the normal derivative on the boundary,  which is called the \emph{constant flow property} in \cite{Sav2016}.

In this paper we show that two-phase overdetermined problems are inherently different. As a matter of fact, due to the introduction of a new degree of freedom (the geometry of the core $D$), we prove that two-phase elliptic overdetermined problems of Serrin-type admit non-symmetric solutions. On the other hand, we show that, for two-phase overdetermined problems of parabolic type,  the stronger assumption of constant heat flow at the boundary for all time $t>0$ leads to radial symmetry (this result holds true even when the overdetermined condition is imposed only on a connected component of the boundary $\pa\Om$). We will also examine another overdetermination, slightly different than the one introduced in \cite{AG1989ARMA}. Namely we will consider the case where, instead of the boundary, the above mentioned constant flow property is satisfied on some fixed surface inside the heat conductor. We will show that, even in this case, the existence of such a surface satisfying the constant flow property leads to the radial symmetry of our heat conductor.

In what follows, we will introduce the notation and the main results of this paper.
Let $\Omega$ be a bounded $C^2$ domain in $\mathbb R^N\ (N \ge 2)$ with boundary $\partial\Omega$, and let $D$ be a bounded  $C^2$ open set in $\mathbb R^N$ which may have finitely many connected components.  Assume that $\Omega\setminus\overline{D}$ is connected and $\overline{D} \subset \Omega$. Denote by $\sigma=\sigma(x)\ (x \in \mathbb R^N)$  the conductivity distribution of the medium given by
$$
\sigma =
\begin{cases}
\sigma_c \quad&\mbox{in } D, \\
\sigma_s \quad&\mbox{in } \Omega \setminus D, \\
\sigma_m \quad &\mbox{in } \mathbb R^N \setminus \Omega,
\end{cases}
$$
where $\sigma_c, \sigma_s, \sigma_m$ are positive constants and $\sigma_c \not=\sigma_s$. This kind of three-phase electrical conductor has been dealt with in \cite{KLS2016} in the study of neutrally coated inclusions.

The first result is a counterexample to radial symmetry for  the following two-phase elliptic overdetermined boundary value problems of Serrin-type: 
\begin{equation}
\label{modified poisson and Hermholtz equation with overdetermined boundary conditions} 
\mbox{\rm div}(\sigma\nabla u) = \beta u -\gamma < 0\ \mbox{ in } \Omega, \quad  u = c \ \mbox{ and } \ \sigma_s \,\pa_\nu u= d_0\  \mbox{ on } \partial \Omega;
\end{equation}
here, $\pa_\nu$ denotes the outward normal derivative at $\pa\Om$, $\beta\ge 0 $, $\gamma >0$, and $c\in\RR$ are given numbers and $d_0$  is some negative constant determined by the data of the problem. 

\begin{theorem}
\label{th:a counterexample to the symmetry for Serrin-type overdetermined problems}
Let $B_R\subset B_1$ be concentric balls of radii $R$ and $1$. For every domain $\Omega$ of class $C^{2,\al}$ sufficiently close to $B_1$,
there exists a 
domain $D$ of class $C^{2,\al}$ (and close to $B_R$) such that problem \eqref{modified poisson and Hermholtz equation with overdetermined boundary conditions} admits a solution for the pair $(D,\Omega)$.  
\end{theorem}
This result is an application of the implicit function theorem. It was shown by Serrin in \cite{Se1971} that, in the one-phase  case ($\si_c=\si_s$), a solution of \eqref{modified poisson and Hermholtz equation with overdetermined boundary conditions} exists if and only if $\Om$ is a ball. Thus, as we shall see for two-phase heat conductors,  Theorem \ref{th:a counterexample to the symmetry for Serrin-type overdetermined problems} 
sets an essential difference between  the parabolic overdetermined regime in Theorem \ref{th:constant flow serrin} and that in the elliptic problem \eqref{modified poisson and Hermholtz equation with overdetermined boundary conditions}. 
\par
A result similar to Theorem \ref{th:a counterexample to the symmetry for Serrin-type overdetermined problems} appeared  in \cite{DEP}, after we completed this paper. That result concerns certain semilinear equations (with a point-dependent nonlinearity) on compact Riemannian manifolds. The techniques used there do not seem to be easily applicable to the two-phase case.

The remaining part of this paper focuses on two-phase overdetermined problems of parabolic type.
The papers \cite{Strieste2016, SBessatsu2017} dealt with the heat diffusion over two-phase or three-phase heat conductors.
Let $u =u(x,t)$ be the unique bounded solution of either the initial-boundary value problem for the diffusion equation:
\begin{eqnarray}
&&u_t =\mbox{ div}( \sigma \nabla u)\quad\mbox{ in }\ \Omega\times (0,+\infty), \label{heat equation initial-boundary}
\\
&&u=1  \ \quad\qquad\qquad\mbox{ on } \partial\Omega\times (0,+\infty), \label{heat Dirichlet}
\\ 
&&u=0  \  \quad\qquad\qquad \mbox{ on } \Omega\times \{0\},\label{heat initial}
\end{eqnarray}
or the Cauchy problem for the diffusion equation:
\begin{equation}
  u_t =\mbox{ div}(\sigma \nabla u)\quad\mbox{ in }\  \mathbb R^N\times (0,+\infty) \ \mbox{ and }\ u\ ={\mathcal X}_{\Omega^c}\ \mbox{ on } \mathbb R^N\times
\{0\},\label{heat Cauchy}
\end{equation}
where ${\mathcal X}_{\Omega^c}$ denotes the characteristic function of the set $\Omega^c=\mathbb R^N \setminus\Omega$. Consider a bounded domain $G$ in $\mathbb R^N$ satisfying
\begin{equation}
\label{near the boundary}
\overline{D} \subset G \subset \overline{G} \subset \Omega\ \mbox{ and } \mbox{ dist}(x,\partial\Omega) \le \mbox{ dist}(x, \overline{D})\ \mbox{ for every } x \in \partial G.
\end{equation}
\begin{figure}[h]
\centering
\includegraphics[width=0.55\textwidth]{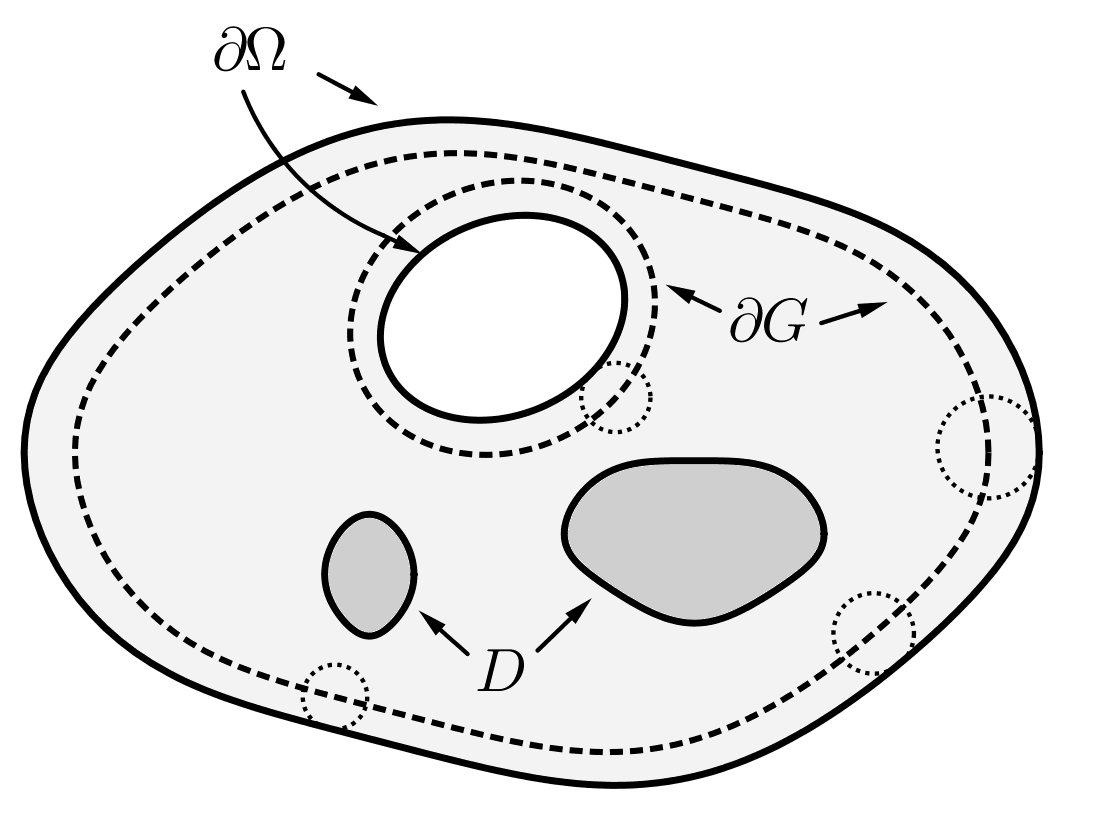}
\caption{The two-phase conductor described by $\Om$ and $D$ and the surface $\pa G$.}
\label{picture0}
\end{figure}
In \cite{Strieste2016, SBessatsu2017},  the third author obtained the following theorems.

\renewcommand*{\thetheorem}{\Alph{theorem}}
\setcounter{theorem}{0}


\begin{theorem}[\cite{Strieste2016}]
\label{th:stationary isothermic} 
Let $u$ be the solution of problem \eqref{heat equation initial-boundary}--\eqref{heat initial},  and let
$\Gamma$ be a connected component of $\partial G$ satisfying
\begin{equation}
\label{nearest component}
\mbox{\rm dist}(\Gamma, \partial\Omega) = \mbox{\rm dist}(\partial G, \partial\Omega).
\end{equation}  
\par
If there exists a function $a : (0, +\infty) \to (0, +\infty) $ satisfying
\begin{equation}
\label{stationary isothermic surface partially}
u(x,t) = a(t)\ \mbox{ for every } (x,t) \in \Gamma \times (0, +\infty),
\end{equation}
then $\Omega$ and $D$ must be concentric balls.
\end{theorem}

\begin{theorem}[\cite{Strieste2016, SBessatsu2017}] 
\label{th:stationary isothermic cauchy} Let $u$ be the solution of problem \eqref{heat Cauchy}. Then the following assertions hold:
\begin{itemize}
\item[\rm (a)] if
there exists a function $a : (0, +\infty) \to (0, +\infty) $ satisfying 
\begin{equation}
\label{stationary isothermic surface}
u(x,t) = a(t)\ \mbox{ for every } (x,t) \in \partial G\times (0, +\infty),
\end{equation}
then $\Omega$ and $D$ must be concentric balls;
\item[\rm (b)] if $\sigma_s=\sigma_m$ and \eqref{stationary isothermic surface partially} holds
on some connected component $\Gamma$ of $\partial G$ satisfying \eqref{nearest component}  for some function $a : (0, +\infty) \to (0, +\infty) $, then  $\Omega$ and $D$ must be concentric balls.
\end{itemize}
\end{theorem}
The condition \eqref{stationary isothermic surface partially} (or  \eqref{stationary isothermic surface})  means that $\Gamma$ (or $\partial G$)  is an isothermic surface  of the normalized temperature $u$ {\it at every time}; for this reason, $\Gamma$ (or $\partial G$) is called a {\it stationary isothermic surface} of $u$.

In this paper, we shall suppose that the solution $u$ of \eqref{heat equation initial-boundary}--\eqref{heat initial} or \eqref{heat Cauchy} admits a  {\it surface $\Ga\subset{\overline\Om}\setminus {\overline D}$ of the constant flow property}, that is there exists a function $d : (0, +\infty) \to \mathbb R $ satisfying
\begin{equation}
\label{constant flow surface partially}
\sigma_s\,\pa_\nu u(x,t) = d(t)\ \mbox{ for every } (x,t) \in \Gamma \times (0, +\infty),
\end{equation}
where $\pa_\nu u$ denotes the outward normal derivative of $u$ at points in $\Gamma$.
\par
We will then prove two types of symmetry results. 
We shall first start with symmetry theorems for solutions that admit a  surface $\Ga$ of the constant flow property in the shell $\Omega\setminus \overline{D}$ of the conductor.


\renewcommand*{\thetheorem}{\arabic{section}.\arabic{theorem}}
\setcounter{theorem}{1}

\begin{theorem}
\label{th:constant flow} 
Let $u$ be the solution of either problem \eqref{heat equation initial-boundary}--\eqref{heat initial} or problem \eqref{heat Cauchy},  and let
$\Gamma$ be a connected component of class $C^2$ of $\partial G$ satisfying \eqref{nearest component}.  
\par
 If there exists a function $d : (0, +\infty) \to \mathbb R $ satisfying \eqref{constant flow surface partially}, then $\Omega$ and $D$ must be concentric balls.
\end{theorem}
With the aid of a simple observation on the initial behavior of the solution $u$ of problem \eqref{heat Cauchy}(see Proposition \ref{prop:the initial limits on the interface}) as in the proof of Theorem \ref{th:constant flow} for problem  \eqref{heat Cauchy}(see  Subsection \ref{subsection 3.3}), Theorems \ref{th:stationary isothermic} and \ref{th:stationary isothermic cauchy} combine to make a single theorem.
\begin{theorem}
\label{th:stationary isothermic surface} Let $u$ be the solution of either problem \eqref{heat equation initial-boundary}--\eqref{heat initial} or problem \eqref{heat Cauchy},  and let
$\Gamma$ be a connected component of $\partial G$ satisfying \eqref{nearest component}. 
\par
If there exists a function $a : (0, +\infty) \to (0, +\infty) $ satisfying \eqref{stationary isothermic surface partially},
then $\Omega$ and $D$ must be concentric balls.
\end{theorem}

A second kind of result concerns multi-phase heat conductors where a connected component of $\partial\Omega$  is a surface of the constant flow property or a stationary isothermic surface. We obtain three symmetry theorems, one for the Cauchy-Dirichlet problem (Theorem \ref{th:constant flow serrin}) and two for the Cauchy problem (Theorems \ref{th:stationary isothermic three-phase} and \ref{th:stationary isothermic on the boundary for cauchy}), with different regularity assumptions.


\begin{theorem}
\label{th:constant flow serrin} Let $u$ be the solution of problem \eqref{heat equation initial-boundary}--\eqref{heat initial},  and let
$\Gamma$ be a connected component of $\partial\Omega$. Suppose that $\Gamma$ is of class $C^6$.  
\par
If there exists a function $d : (0, +\infty) \to \mathbb R $ satisfying \eqref{constant flow surface partially},
then $\Omega$ and $D$ must be concentric balls. 
\end{theorem}

When $D=\varnothing$, $\Ga=\pa\Omega$ and $\sigma$ is constant on $\mathbb R^N$,  the same overdetermined boundary condition of Theorem \ref{th:constant flow serrin} has been introduced in \cite{AG1989ARMA, GS2001pams} and similar symmetry theorems have been proved by the method of moving planes introduced by \cite{Se1971} and \cite{Alek1958vestnik}.  Theorem  \ref{th:constant flow serrin} gives a new symmetry result for two-phase heat conductors, in which that method cannot be applied.
Recently,  an analogous problem was re-considered in  \cite{Sav2016} in the context of the heat flow in smooth Riemannian manifolds: it was shown that the same overdetermined boundary condition implies that $\pa\Omega$ must be an isoparametric surface (and hence $\pa \Omega$ is a sphere if compactness is assumed). We remark that the methods introduced in \cite{Sav2016} cannot be directly applied to our two-phase setting due to a lack of regularity.    


\begin{theorem}
\label{th:stationary isothermic three-phase} Let $u$ be the solution of problem \eqref{heat Cauchy},  and let
$\Gamma$ be a connected component of $\partial\Omega$. Suppose that $\Gamma$ is of class $C^6$.  
\par
If there exists a function  $a : (0, +\infty) \to (0, +\infty) $ satisfying \eqref{stationary isothermic surface partially},
then $\Omega$ and $D$ must be concentric balls. 
\end{theorem}

The $C^6$-regularity assumption of Theorems \ref{th:constant flow serrin} and \ref{th:stationary isothermic three-phase} does not seem very optimal, but it is needed to construct the barriers where we use the fourth derivatives of the distance function to the boundary.  It can instead be removed for
problem \eqref{heat Cauchy}, in the particular the case in which  $\sigma_s=\sigma_m$. This can be done by complementing the proof of Theorem \ref{th:constant flow serrin} with the techniques developed in \cite{MPS2006tams}.


\begin{theorem}
\label{th:stationary isothermic on the boundary for cauchy} Set $\sigma_s=\sigma_m$ and let $u$ be the solution of problem \eqref{heat Cauchy}.  Let $\Gamma$ be a connected component of $\partial\Omega$.
\begin{itemize}
\item[\rm (a)] If there exists a function $a : (0, +\infty) \to (0,+\infty)$ satisfying \eqref{stationary isothermic surface partially}, then  $\Omega$ and $D$ must be concentric balls.
\item[\rm (b)] If $N\ge 3$, suppose that $\Gamma$ is strictly convex. If there exists a function $d : (0, +\infty) \to \mathbb R$ satisfying \eqref{constant flow surface partially},  then  $\Omega$ and $D$ must be concentric balls.
\end{itemize}
\end{theorem}

\vskip 2ex
The rest of the paper is organized as follows. 
Section \ref{section6} is devoted to the proof of Theorem \ref{th:a counterexample to the symmetry for Serrin-type overdetermined problems}, which is a combination of the implicit function theorem and techniques pertaining to the realm of {\it shape optimization}.
In Section \ref{section2} we give some preliminary notations and recall some useful results from \cite{Strieste2016, SBessatsu2017}. In Section \ref{section3}, we shall carry out the proofs of Theorems \ref{th:constant flow} and \ref{th:stationary isothermic surface}, based on a balance law, the short-time behaviour of the solution, and on the study of a related elliptic problem. The proof of Theorem \ref{th:constant flow serrin} will be performed in Section \ref{section4}: the relevant parabolic problem will be converted into a family of elliptic ones, by a Laplace transform, and new suitable barriers controlled by geometric parameters of the conductor  will be constructed for the transformed problem.  The same techniques will also be used in Subsection \ref{subsection4.5} to prove Theorem \ref{th:stationary isothermic three-phase}. Section \ref{section5} contains the proof of Theorem \ref{th:stationary isothermic on the boundary for cauchy}: here, due to the more favorable structure of the Cauchy problem in hand, we are able to use the techniques of \cite{MPS2006tams} to obtain geometrical information.

\setcounter{equation}{0}
\setcounter{theorem}{0}

\section{Non-uniqueness for a two-phase Serrin's problem}
\label{section6}

Here, the proof of Theorem \ref{th:a counterexample to the symmetry for Serrin-type overdetermined problems} will be obtained by a perturbation argument.

Let $D$, $\Omega\subset\rn$ be two bounded domains of class $C^{2,\al}$ with $\overline{D}\subset \Omega$. 
We look for a pair $(D,\Om)$ for which the overdetermined problem \eqref{modified poisson and Hermholtz equation with overdetermined boundary conditions}  has a solution for some negative constant $d_0$. By evident normalizations, it is sufficient to examine \eqref{modified poisson and Hermholtz equation with overdetermined boundary conditions} with $\sg_s=1$ in the form
\begin{eqnarray}
&&\dv(\sigma \nabla u)=\beta u - \gamma<0\quad\mbox{ in }\ \Omega, \label{pbcava eq}
\\
&&u=0  \ \quad\qquad\qquad\qquad\qquad \mbox{ on } \partial\Omega, \label{pbcava dirichlet}
\\ 
&&\partial_\nu u=-\La  \  \quad\qquad\qquad\qquad \mbox{ on } \pa\Omega,\label{pbcava neumann}
\end{eqnarray}
where  $\be\ge 0$, $\ga>0$, and $\sg=\sg_c {\mathcal X}_D+{\mathcal X}_{\Om\setminus D}$.
By the divergence theorem, the constant $\La$ is related to the other data of the problem by the formula:
\begin{equation}
\label{whatisd}
\La=\frac1{|\pa\Om|}\left\{\ga\,|\Om|-\beta\,\int_{\Om} u\,dx\right\};
\end{equation}
here, the bars indifferently denote the volume of $\Om$ and the $(N-1)$-dimensional Hausdorff measure of $\partial\Omega$.
\par
It is obvious that, for all values of $\sg_c>0$, the pair $(B_R, B_1)$ in the assumptions of the theorem is a solution to the overdetermined problem \eqref{pbcava eq}--\eqref{pbcava neumann} for some $\La$.
We will look for other solution pairs of \eqref{pbcava eq}--\eqref{pbcava neumann} near $(B_R, B_1)$ by a perturbation argument which is based on the following version of the implicit function theorem, for the proof of which we refer to \cite[Theorem 2.7.2, pp. 34--36]{Nams2001}.


\renewcommand*{\thetheorem}{\Alph{theorem}}
\setcounter{theorem}{2}


\begin{theorem}[Implicit function theorem]
\label{ift}
Suppose that $\mathcal{F}$, $\mathcal{G}$ and $\mathcal{H}$ are three Banach spaces, $U$ is an open subset of  $\mathcal{F}\times\mathcal{G}$, $(f_0,g_0)\in U$, and $\Psi:U\to\mathcal{H}$ is a Fr\'echet differentiable mapping such that $\Psi(f_0,g_0)=0$. Assume that the partial derivative $\pa_f\Psi(f_0,g_0)$ of $\Psi$ with respect to $f$ at $(f_0,g_0)$ is a bounded invertible linear transformation from  $\cF$ to $\cH$. 
\par
Then there exists an open neighborhood $U_0$ of $g_0$ in $\mathcal{G}$ such that
there exists a unique Fr\'echet differentiable function $f:U_0\to \mathcal{F}$ such that $f(g_0)=f_0$, $(f(g),g)\in U$ and $\Psi(f(g),g)=0$ for all $g\in U_0$.  
\end{theorem}

\subsection{Preliminaries}
We introduce the functional setting for the proof of Theorem \ref{th:a counterexample to the symmetry for Serrin-type overdetermined problems}. Set $D=B_R$ and $\Om=B_1$. For $\alpha\in(0,1)$,  let
$\phi\in C^{2,\al}(\rn, \rn)$ satisfy that  $\mathrm{Id}+\phi$  is a  diffeomorphism from $\rn$ to $\rn$,  and
$$
\phi=f\,\nu \ \mbox{ on } \ \pa D \quad \mbox{ and } \quad \phi=g\,\nu \ \mbox{ on } \ \pa\Om,
$$
where $\mathrm{Id}$ denotes the identity mapping, $f$ and $g$ are given functions of class $C^{2,\al}$ on $\pa D$ and $\pa\Om$, respectively, and $\nu$ indistinctly denotes the outward unit normal to both $\pa D$ and $\pa\Om$.
Next, we define the sets
$$
\Om_g=(\mathrm{Id}+\phi)(\Om) \ \mbox{ and } \ D_f=(\mathrm{Id}+\phi)(D).
$$ 
If $f$ and $g$ are sufficiently small, $D_f$ and $\Om_g$ are such that $\ol{D_f}\subset\Om_g$. 
\par
Now, we consider the Banach spaces (equipped with their standard norms):
\begin{eqnarray*}
& \cF=\Bigl\{f\in C^{2,\al}(\pa D): \int_{\pa D} f\, dS =0\Bigr\},\quad 
\cG=\Bigl\{g\in C^{2,\al}(\pa\Om): \int_{\pa\Om} g\, dS =0\Bigr\}, \\
&\cH=\Bigl\{h\in C^{1,\al}(\pa\Om) : \int_{\pa\Om} h\, dS =0\Bigr\}.
\end{eqnarray*}
In order to be able to use Theorem \ref{ift}, we introduce a mapping $\Psi: \cF\times \cG\to \cH$ by:
\begin{equation}
\Psi(f,g)=\left\{\pa_{\nu_g} u_{f,g} + \La_{f,g}\right\} J_\tau({g}) \ \mbox{ for } \ (f,g)\in\cF\times\cG.
\end{equation}
Here, $u_{f,g}$ is the solution of  \eqref{pbcava eq}--\eqref{pbcava dirichlet} with $\Om=\Om_g$ and $\sg=\sg_c\,\cX_{D_f}+\cX_{\Om_g\setminus D_f}$, $\nu_g$ stands for the outward unit normal to $\pa \Om_g$, and $\La_{f,g}$ is computed via \eqref{whatisd}, with $\Om=\Om_g$ and $u=u_{f,g}$. Also, by a slight abuse of notation, $\pa_{\nu_g} u_{f,g}$ means
the function of value 
$$
\gr u_{f,g}(x+g(x)\,\nu(x)) \cdot \nu_g(x+g(x)\,\nu(x))\ \mbox{ at any } x\in\pa\Om,
$$
where $\nu$ is the outward unit normal to $\pa\Om$.
Finally, the term $J_\tau(g)>0$ is the tangential Jacobian associated to the transformation $x\mapsto x+g(x)\,\nu(x)$ (see \cite[Definition 5.4.2, p. 190]{henrot}): this term ensures that the image $\Psi(f,g)$ has zero integral over $\pa\Om$ for all $(f,g)\in \cF\times \cG$, as an integration of \eqref{pbcava neumann} on $\pa\Om_g$ requires, when $\La=\La_{f,g}$.
\par
Thus, by definition, we have $\Psi(f,g)=0$ if and only if the pair $(D_f,\Om_g)$ solves \eqref{pbcava eq}--\eqref{pbcava neumann}. Moreover, we know that the mapping $\Psi$ vanishes at $(f_0, g_0)=(0,0)$.

\subsection{Computing the derivative of $\Psi$}
The Fr\'echet differentiability of $\Psi$ in a neighborhood of $(0,0)\in\mathcal{F}\times\mathcal{G}$ can be proved, in a standard way, by following the proof of \cite[Theorem 5.3.2, pp. 183--184]{henrot},  with the help of the regularity theory for elliptic operators with piecewise constant coefficients.  In particular, the H\"older continuity of the first and second derivatives  of the function $u_{f,g}$ up to  the interface $\partial D_f$, which is stated in \cite[Theorem 16.2, p. 222]{LaU1968},   is obtained by flattening the interface with a diffeomorphism of class $C^{2,\alpha}$ as in \cite[Chapter 4, Section 16, pp. 205--223]{LaU1968} or in \cite[Appendix, pp. 894--900]{DiBEF1986na} and by using the classical regularity theory for linear elliptic partial differential equations (\cite{LaU1968, giaquinta, ACM2019}).


We will now proceed to the actual computation of $\pa_f \Psi (0,0)$. Since $\Psi$ is Fr\'echet differentiable, $\pa_f \Psi (0,0)$ can be computed as a G\^ateaux derivative:
$$
\pa_f\Psi(0,0)(f)= \lim_{t\to 0} \frac{\Psi(t f,0)-\Psi(0,0)}{t} \ \mbox{ for } \ f\in\cF.
$$
\par
From now on, we fix $f\in\cF$, set $g=0$ and, to simplify notations, we will write $D_t, u_t, \La(t)$ in place of $D_{tf}, u_{tf,0}, \La_{tf,0}$; in this way, we can agree that $D_0=D$, $u_0=u$, and so on.  Also,
in order to carry out our computations, we introduce some standard notations, in accordance with \cite{henrot} and \cite{SG}: the {\it shape derivative} of $u$ is defined by
\begin{equation}\label{def shape der of state function}
u'(x)=\restr{\frac{d}{dt}}{t=0} u_t(x)  \ \text{ for } \ x\in\Om.
\end{equation}

\renewcommand*{\thetheorem}{\arabic{section}.\arabic{theorem}}
\setcounter{theorem}{0}

In particular, we will employ the use of the following characterization of the shape derivative $u'$ of $u$. We refer to \cite[Proposition 2.3]{cava} where the case $\be=0$ is analyzed, and to \cite[Theorem 2.5]{DaKa11} where $\be<0$ is an eigenvalue. The case $\be>0$ can be treated analogously and therefore the proof will be omitted. 

 
\begin{lemma}
For every $f\in\cF$, the shape derivative $u'$ of $u_t$ solves the following:
\begin{eqnarray}
\label{pbu' eq}& \sg \De u'=\beta u' &\quad\mbox{ in }\ D\cup (\Om\setminus \overline{D}), \\
\label{pbu' flux}& [\sg \pa_\nu u']=0 &\quad \mbox{ on } \pa D,\\
\label{pbu' jump}& [u']=-[\pa_\nu u]f & \quad \mbox{ on }\pa D,\\ 
\label{pbu' bdary}&u'=0  &\quad\mbox{ on }\ \partial\Om.     
\end{eqnarray}
\end{lemma}
In the above, we used square brackets to denote the jump of a function across the interface $\pa D$. More 
precisely, for any function $\varphi$ we mean $[ \varphi ]= \varphi_+-\varphi_-$, where the subscripts $+$ and $-$ denote the relevant quantities in the two phases $\Om\setminus \overline{D}$ and $D$ respectively and the equality here is understood in the classical sense.

\begin{lemma}
\label{lem:shape-derivatives}
For all $f\in\cF$ we have 
$\La'(0)=0$.
\end{lemma}

\begin{proof}
We rewrite \eqref{whatisd} as
$$
\La(t)|\pa \Om|-\ga|\Om|=-\beta \int_{\Omega} u_t \, dS,
$$
then differentiate and evaluate at $t=0$. The derivative of the left-hand side equals 
$\La'(0)\,|\pa\Om|$.
Thus, we are left to prove that the derivative of the function defined by
$$
I(t)=\int_{\Omega} u_t\, dx
$$
is zero at $t=0$.
\par
To this aim, since $u_t$ solves \eqref{pbcava eq} for $D=D_t$, we multiply both sides of this for $u_t$ and integrate to obtain that
\begin{equation}\nonumber
\ga\,I(t)=\gamma \int_{\Omega} u_t\, dx = \beta \int_{\Om} u_t^2\, dx+ \sg_c\int_{D_t} \, \abs{\gr u_t \,}^2 dx+\int_{\Omega\setminus\overline{D_t}} \abs{\gr u_t \,}^2 dx,
\end{equation}
after an integration by parts.
Thus, the desired derivative can be computed by using Hadamard's formula (see \cite[Corollary 5.2.8, p. 176]{henrot}):
\begin{equation*}
\label{gamma int u prime}
\begin{aligned}
\ga\,I'(0)&=2\beta \int_{\Omega} u u' \,dx 
+2\int_\Om \sg \gr u \cdot \gr u' \, dx
+ \sg_c\int_{\partial D} (\partial_\nu u_-)^2 f \, dS
-\int_{\pa D} (\partial_\nu u_+)^2 f \, dS 
\\
   &=
2\beta \int_{\Omega} u u' \,dx  
+  2 \int _{\Omega} \sg \gr u \cdot \gr u' \,dx 
=0.
\end{aligned}
\end{equation*}
Here, in the second equality we used that $\pa_\nu u_{-}$ and $\pa_\nu u_+$ are constant on $\pa D$ and that $f\in\cF$, while, the third equality ensues by integrating \eqref{pbu' eq} against $u$.
\end{proof}

\begin{theorem}
\label{prop psi'}
The Fr\'echet derivative $\pa_f \Psi (0,0)$ defines a mapping from $\cF$ to $\cH$ by the formula
$$
\pa_f\Psi (0,0)(f)=\pa_\nu u',
$$
where $u'$  is the solution of the boundary value problem \eqref{pbu' eq}--\eqref{pbu' bdary}.
\end{theorem}
\begin{proof}
Since $\Psi$ is Fr\'echet differentiable, we can compute $\pa_f \Psi$ as a G\^ateaux derivative as follows:
$$
\partial_f\Psi(0,0)(f)=\restr{\frac{d}{dt}}{t=0}\Psi(t f,0)=\restr{\frac{d}{dt}}{t=0}\left\{ \gr u_t(x)\cdot\nu(x) + \La(t)  \right\}J_\tau (0).
$$
Since $J_\tau (0)=1$, the thesis is a direct consequence of Lemma \ref{lem:shape-derivatives} and definition \eqref{def shape der of state function}.
Finally, the fact that this mapping is well-defined (i.e. $\partial_\nu u'$ actually belongs to $\mathcal{H}$ for all $f\in\mathcal{F}$) follows from the calculation 
$$
\int_{\partial\Omega} \partial_\nu u' \,dS= \int_{\Omega} \dv(\sg \gr u')\,dx=\beta \int_{\Omega} u'\,dx=
\be\, I'(0)=0,
$$ 
where we also used \eqref{pbu' eq}--\eqref{pbu' bdary}.
\end{proof}

\subsection{Applying the implicit function theorem}

The following result clearly implies Theorem \ref{th:a counterexample to the symmetry for Serrin-type overdetermined problems}.


\begin{theorem}\label{mainthm cava}
There exists $\varepsilon>0$ such that, for all $g\in\mathcal{G}$ with $\norm{g}<\varepsilon$ there exists a unique $f(g)\in\mathcal{F}$ such that the pair $(D_{f(g)},\Omega_{g})$ is a solution of the overdetermined problem \eqref{pbcava eq}--\eqref{pbcava neumann}.
\end{theorem}
\begin{proof}
This theorem consists of a direct application of Theorem \ref{ift}.
We know that the mapping $(f,g)\mapsto\Psi(f,g)$ is Fr\'echet differentiable and we computed its Fr\'echet derivative with respect to the variable $f$ in Theorem \ref{prop psi'}. We are left to prove that the mapping $\pa_f\Psi(0,0):\cF\to\cH$, given in Theorem \ref{prop psi'}, 
is a bounded and invertible linear transformation.
\par
Linearity and boundedness of $\pa_f\Psi(0,0)$ ensue from the properties of problem \eqref{pbu' eq}--\eqref{pbu' bdary}. 

We are now going to prove the invertibility of $\pa_f \Psi(0,0)$. To this end we study the relationship between the spherical harmonic expansions of the functions $f$ and $u'$ (we refer to \cite[Section 4]{cava} where the same technique has been exposed in detail). Suppose that, for some real coefficients $\al_{k,i}$ the following holds
\begin{equation}\label{f in sphar}
f(R\theta) = \sum_{k=1}^\infty \sum_{i=1}^{d_k} \al_{k,i} Y_{k,i}(\theta), \quad \mbox{ for } \theta\in\mathbb{S}^{N-1}. 
\end{equation} 
Here $Y_{k,i}$ denotes the solution of the eigenvalue problem $-\Delta_{\mathbb{S}^{N-1}}Y_{k,i}=\lambda_k Y_{k,i}$ on $\mathbb{S}^{N-1}$, with $k$-th eigenvalue $\lambda_k=k(N+k-2)$ of multiplicity $d_k$.
Under the assumption \eqref{f in sphar}, we can apply the method of separation of variables to get
\begin{equation}\label{u' in sphar}
u'(r\theta)= \sum_{k=1}^\infty\sum_{i=1}^{d_k}\al_{k,i}s_k(r)Y_{k,i}(\theta), \quad \mbox{ for }r\in (0,R)\cup (R,1) \mbox{ and } \theta \in\mathbb{S}^{N-1}.
\end{equation}
Here $s_k$ denotes the solution of the following problem:
\begin{eqnarray}
&&\sg\left\{ \pa_{rr} s_k+ \frac{N-1}{r}\,\pa_r s_k - \frac{k(k+N-2)}{r^2}s_k \right\} = \be s_k \quad\mbox{ in } (0,R)\cup (R,1), \label{the 2nd order ODE}\\
&& s_k(R^+)-s_k(R^-)=\pa_r u(R^-)-\pa_r u(R^+), \quad \sg_c\, \pa_r s_k(R^-)= \pa_r s_k(R^+), \nonumber\\ 
&& s_k(1)=0,\qquad  \pa_r s_k(0)=0, \nonumber
\end{eqnarray}
where, by a slight abuse of notation, the letters $\sg$ and $u$ mean the radial functions
$\sg(\abs{x})$ and $u(\abs{x})$ respectively. By \eqref{u' in sphar} we see that $\pa_f \Psi(0,0)$ preserves the eigenspaces of the Laplace--Beltrami operator, and in particular, $\pa_f\Psi(0,0)$ is invertible if and only if $\pa_r s_k(1)\ne0$ for all $k\in\{1,2,\dots\}$. Let us show the latter. Suppose by contradiction that $\pa_r s_k(1)=0$ for some $k\in\{1,2,\dots\}$. Then, since $s_k(1)=0$, by the unique solvability of the Cauchy problem for the ordinary differential equation \eqref{the 2nd order ODE}, 
$s_k \equiv 0$ on the interval $[R,1]$. Hence $\pa_r s_k(R^-)=0$. Multiplying \eqref{the 2nd order ODE} by $r^2$ and letting $r \to 0$ yield that $s_k(0) = 0$. Therefore,  since $\beta \ge 0$,  assuming that $s_k$ achieves either its positive maximum or its negative minimum at a point in the interval $(0,R]$ contradicts equation  \eqref{the 2nd order ODE}. Thus  $s_k \equiv 0$ also on $[0,R]$.
On the other hand, since $\sg_c\ne1$, we see that  $\pa_\nu u_+- \pa_\nu u_-\ne 0$ on $\pa D$ and hence $s_k(R^-)\not=0$, which is a contradiction. \end{proof}

\setcounter{equation}{0}

\section{Preliminaries for overdetermined parabolic problems}
\label{section2}

In this section, we introduce some notations and recall the results obtained in \cite{Strieste2016, SBessatsu2017} that will be useful in the sequel.

For a point $x \in \mathbb R^N$ and a number $r > 0$, we set:
$
B_r(x) = \{ y \in \mathbb R^N\ :\ |y-x| < r \}.
$
Also, for a bounded $C^2$ domain $\Omega\subset\mathbb R^N$, $\kappa_1(y),\dots,\kappa_{N-1}(y)$ will always denote the principal curvatures of $\partial\Omega$ at a point $y\in\partial\Omega$ with 
respect to the inward normal direction to $\partial\Omega$. Then, we set
\begin{equation}
\label{product-curvatures}
\Pi_{\partial\Omega}(r, y)=\prod\limits_{j=1}^{N-1}\bigl[1/r - \kappa_j(y)\bigr] \ \mbox{ for } \ y\in\partial\Omega \mbox{ and } r >0.
\end{equation}
Notice that, if $B_r(x)\subset\Omega$ and $\overline{B_r(x)} \cap \partial\Omega = \{ y \}$ for some $y \in \partial\Omega$, then $\kappa_j(y) \le 1/r$ 
for all $j$'s, and hence $\Pi_{\partial\Omega}(r, y)\ge 0$.

The initial behavior of the heat content of such kind of ball is controlled by the geometry of the domain, as the following proposition explains.

\renewcommand*{\thetheorem}{\Alph{theorem}}
\setcounter{theorem}{3}

\begin{proposition}[{\cite[Proposition 2.2, pp. 171--172]{Strieste2016}}]
\label{prop:heat content asymptotics} 
Let $x \in \Omega$ and assume that $B_r(x)\subset\Omega$ and $\overline{B_r(x)} \cap \partial\Omega = \{ y \}$ for some $y \in \partial\Omega$. Let $u$ be the solution of either problem \eqref{heat equation initial-boundary}--\eqref{heat initial} or problem \eqref{heat Cauchy}. 
\par
Then we have:
\begin{equation}
\label{asymptotics and curvatures}
\lim_{t\to +0}t^{-\frac{N+1}4 }\!\!\!\int\limits_{B_r(x)}\! u(z,t)\ dz=
\frac{C(N, \sigma)}{\sqrt{\Pi_{\pa\Om}(r,y)}}.
\end{equation}
Here, $C(N, \sigma)$ is the positive constant given by
$$
C(N,\sigma) = \left\{\begin{array}{rll}2\sigma_s^{\frac {N+1}4}c(N) \ &\mbox{ for problem \eqref{heat equation initial-boundary}--\eqref{heat initial} },
\\
\frac {2\sqrt{\sigma_m}}{\sqrt{\sigma_s}+\sqrt{\sigma_m}}\sigma_s^{\frac {N+1}4}c(N) &\mbox{ for problem \eqref{heat Cauchy} },
\end{array}\right.
$$
where $c(N)$ is a positive constant only depending on $N$. 
\par
When $\kappa_j(y) = 1/r$ for some $j \in \{ 1, \cdots, N-1\}$, 
\eqref{asymptotics and curvatures} holds by setting its right-hand side to $+\infty$
\end{proposition}
Notice that, if $\sigma_s=\sigma_m$, the constant for problem \eqref{heat Cauchy} is just half of that  for problem \eqref{heat equation initial-boundary}--\eqref{heat initial}.
\par
By examining the proof of Proposition \ref{prop:heat content asymptotics} given in \cite{Strieste2016}, we can also specify the initial behavior of the solution of problem \eqref{heat Cauchy}.

\begin{proposition}[\cite{Strieste2016}]
\label{prop:the initial limits on the interface} 
As $t \to +0$, the solution $u$ of  problem \eqref{heat Cauchy} converges to the number $\frac {\sqrt{\sigma_m}}{\sqrt{\sigma_s}+\sqrt{\sigma_m}}$, uniformly on $\partial\Omega$.
\end{proposition}

\begin{proof}
We refer to \cite{Strieste2016} for the relevant notations and formulas.
In fact, the inequalities \cite[(22), p. 174]{Strieste2016} yield in particular that
\begin{multline*}
(1-\varepsilon)\frac \mu{\theta_-} F_-(0) -2E_1e^{-\frac {E_2}t} \le u(x,t) \le (1+\varepsilon)\frac \mu{\theta_+} F_+(0) +2E_1e^{-\frac {E_2}t}\\ 
\mbox{ for every } (x,t) \in \partial\Omega \times (0,t_\varepsilon].
\end{multline*}
Thus,
$$
(1-\varepsilon)\frac \mu{\theta_-} F_-(0) \le \liminf_{t\to 0^+} u(x,t) \le \limsup_{t\to 0^+} u(x,t) \le (1+\varepsilon)\frac \mu{\theta_+} F_+(0)
$$
for every $\ve>0$, and hence our claim follows by observing that
$$
(1-\varepsilon)\frac \mu{\theta_-} F_-(0) \ \mbox{ and }
(1+\varepsilon)\frac \mu{\theta_+} F_+(0) \to \frac {\sqrt{\sigma_m}}{\sqrt{\sigma_s}+\sqrt{\sigma_m}}\ \mbox{ as }
\varepsilon \to +0,
$$
since both $F_-(0)$ and $F_+(0)$ converge to $F(0) =\frac 12$ as $\varepsilon \to +0$.
\end{proof}

\medskip

We conclude this section by recalling two results from \cite{SBessatsu2017}. The first one is a lemma 
that, for an elliptic equation, states the uniqueness of the reconstruction of the conductivity $\si$  from boundary measurements.

\begin{lemma}[{\cite[Lemma 3.1]{SBessatsu2017}}]
\label{le: the unique determination}
 Let $\Omega$ be a bounded $C^2$-regular domain in $\mathbb R^N\ (N \ge 2)$ with boundary $\partial\Omega$. Let $D_1$ and $D_2$ be two, possibly empty, bounded Lipschitz open sets, each of which may have finitely many connected components. Assume that $D_1 \subset D_2 \subset \overline{D_2} \subset \Omega$ and that both $\Omega\setminus\overline{D_1}$ and  $\Omega\setminus\overline{D_2}$ are connected. 
\par
Let $\sigma_j :\Om\to\RR$ $(j=1,2)$ be given by
$$
\sigma_j =
\begin{cases}
\sigma_c \quad&\mbox{in } D_j, \\
\sigma_s \quad&\mbox{in } \Omega \setminus D_j,
\end{cases}
$$
where $\sigma_c, \sigma_s$ are positive constants with $\sigma_c \not=\sigma_s$.  
\par
For a non-zero function $g \in L^2(\partial\Omega)$, let $v_j \in H^1(\Omega)$ $(j=1,2)$ satisfy
\begin{equation}
\label{modified poisson equation}
\mbox{\rm div}(\sigma_j\nabla v_j) = v_j -1\ \mbox{ in } \Omega\ \mbox{ and }\  \sigma_s  {\partial_\nu v_j} = g \ \mbox{ on } \partial\Omega.
\end{equation}
\par
If $v_1=v_2$  on $\partial\Omega$, then $v_1=v_2$ in $\Omega$  and $D_1=D_2$.
\end{lemma} 

The second result from \cite{SBessatsu2017} gives symmetry in a two-phase overdetermined problem of Serrin type in a special regime. Some preliminary notation is needed. We
let $D$ be a bounded open set of class $C^2$, which may have finitely many connected components, 
compactly contained in a ball $B_r(x)$ and such that $B_r(x)\setminus\overline{D}$ is connected. Also, we denote by $\sigma: B_r(x)\to\RR$  the conductivity distribution given by
$$
\sigma =
\begin{cases}
\sigma_c \quad&\mbox{in } D, \\
\sigma_s \quad&\mbox{in } B_r(x) \setminus D,
\end{cases}
$$
where $\sigma_c, \sigma_s$ are positive constants and $\sigma_c \not=\sigma_s$. 
\begin{theorem}[{\cite[Theorem 5.1]{SBessatsu2017}}]
\label{th:constant Neumann boundary condition} 
Let $v \in H^1(B_r(x))$ be the unique solution of the following boundary value problem:
\begin{equation}
\label{modified poisson and Hermholtz equation}
\mbox{\rm div}(\sigma\nabla v) = \be v -\ga < 0\ \mbox{ in } B_r(x)\ \mbox{ and }\  v = c \ \mbox{ on } \partial B_r(x),
\end{equation}
where $\be\ge 0, \ga > 0$ and $c$ are real constants. 
\par
If $v$ satisfies
\begin{equation}
\label{overdetermined Neumann 1}
\sigma_s\,\pa_\nu v= d\ \mbox{ on } \partial B_r(x),
\end{equation}
for some negative constant $d$, then $D$ must be a ball centered at $x$.
\end{theorem}

\renewcommand*{\thetheorem}{\arabic{section}.\arabic{theorem}}
\setcounter{theorem}{0}

\setcounter{equation}{0}

\section{The constant flow property in the shell}
\label{section3}

In this section, we will carry out the proof of Theorem \ref{th:constant flow}. 

\subsection{Preliminary lemmas}

We start by a lemma that informs on the rough short-time asymptotic behavior of the solution of either \eqref{heat equation initial-boundary}--\eqref{heat initial} or \eqref{heat Cauchy} away from $\pa\Om$.
For $\rho > 0$, we use the following notations:
$$
\Omega_\rho = \{ x \in \Omega\ :\mbox{ dist}(x,\partial\Omega) \ge \rho \}\ \mbox{ and }\ \Omega_\rho^c
=\{ x \in \mathbb R^N\setminus\Omega\ : \mbox{ dist}(x,\partial\Omega) \ge \rho \}.
$$

\begin{lemma} 
\label{le:initial behavior and decay at infinity} 
Let $u$ be the solution of either problem \eqref{heat equation initial-boundary}--\eqref{heat initial} or \eqref{heat Cauchy}. 
\begin{itemize}
\item[\rm (1)]  The following inequalities hold:
$$
0 < u(x,t) < 1  \quad \mbox{ for every }(x,t) \in \Omega\times (0,+\infty) \mbox{ or } (x,t)\in \mathbb R^N \times (0,+\infty).
$$
\item[\rm (2)] For every $\rho > 0$, there exist two positive constants $B$ and $b$ such that
$$
0 < u(x,t) < B e^{-\frac bt}\quad \mbox{ for every } (x,t) \in \Omega_\rho \times (0,+\infty)
$$
and, moreover, if $u$ is the solution of \eqref{heat Cauchy}, then
$$
0 < 1-u(x,t) < B e^{-\frac bt}\quad \mbox{ for every } (x,t) \in \Omega_\rho^c \times (0,+\infty).
$$
Here $B$ and $b$ depend only on $N, \sigma_c, \sigma_s, \sigma_m$ and $\rho$.
 \item[\rm(3)] The solution $u$ of \eqref{heat Cauchy} is such that 
 $$
 \lim\limits_{|x| \to \infty} (1-u(x,t)) = 0  \quad \mbox{ for every } t \in (0,+\infty).
 $$
 \end{itemize}
\end{lemma}

\begin{proof}
 Claim (1) follows from the strong comparison principle. 
\par
To prove (2) and (3), we make use of the Gaussian bounds for the fundamental solutions of parabolic equations due to
Aronson \cite[Theorem 1, p. 891]{Ar1967bams}(see also \cite[p. 328]{FaS1986arma}). In fact, if $g = g(x,\xi,t)$ is the fundamental solution of \eqref{heat equation initial-boundary}, there exist two positive constants $\alpha$ and $M$ depending only on  $N, \sigma_c, \sigma_s$ and $\sigma_m$ such that
\begin{equation}
\label{Gaussian bounds}
M^{-1}t^{-\frac N2}e^{-\frac{\alpha|x-\xi|^2}{t}}\le g(x,\xi,t) \le Mt^{-\frac N2}e^{-\frac{|x-\xi|^2}{\alpha t}} 
\end{equation}
 for all $x, \xi \in \mathbb R^N$ and $t \in (0,+\infty)$. 
\par
When $u$ is the solution of \eqref{heat Cauchy},  $1-u$ can be regarded as the unique bounded solution of \eqref{heat Cauchy} with initial data ${\mathcal X}_{\Omega}$ in place of $\cX_{\Om^c}$. Hence we have from \eqref{Gaussian bounds}:
 $$
 1-u(x,t) = \int\limits_{\mathbb R^N}  g(x,\xi,t){\mathcal X}_{\Omega}(\xi)\ d\xi \le M t^{-\frac N2}\int\limits_{\Omega} e^{-\frac{|x-\xi|^2}{\alpha t}} d\xi.
 $$
Since $|x-\xi| \ge \rho$ for every $x \in \Omega_\rho^c$ and $\xi\in\Omega$, it follows that 
$$
 t^{-\frac N2}\int\limits_{\Omega} e^{-\frac{|x-\xi|^2}{\alpha t}} d\xi \le e^{-\frac{\rho^2}{2\alpha t}}t^{-\frac N2}\int\limits_{\Omega} e^{-\frac{|x-\xi|^2}{2\alpha t}} d\xi \le (2\pi\alpha)^{\frac N2}e^{-\frac{\rho^2}{2\alpha t}},
$$
for every $x \in \Omega_\rho^c$, being $\Om\subset\RR^N$.
Thus, for any fixed $\rho>0$, the solution $u$ of \eqref{heat Cauchy} satisfies the inequality
 \begin{equation*}
 \label{decay at infinity Cauchy problem}
  1-u(x,t) \le M(2\pi\alpha)^{\frac N2}e^{-\frac{\rho^2}{2\alpha t}}\ \mbox{ for every } (x,t) \in \Omega_\rho^c\times (0,+\infty),
 \end{equation*}
which yields the second formula of (2), with $B=M\,(2\pi\alpha)^{\frac N2}$ and $b=\rho^2/2\al$, and (3),  by the arbitrariness of $\rho$. 
\par
The first formula of (2) certainly holds for  $t \in (1,+\infty)$, if we choose $B > 0$ so large as to have  that $B e^{-b} \ge 1$, since (1) holds.  Therefore, it suffices to consider the case in which $t \in (0,1]$.
\par
Let $\rho > 0$, set
$$
 \mathcal N = \{ x \in \mathbb R^N : \mbox{ dist}(x, \partial\Omega) < \rho/2 \},
$$
and define $v = v(x,t)$ by 
$$
 v(x,t) = \mu \int\limits_{\mathcal N} g(x,\xi,t)\ d\xi\quad \mbox{ for every }(x,t) \in \mathbb R^N\times(0,+\infty).
$$
Notice that $v$ is the unique bounded solution of 
$$
v_t = \mbox{ div}(\sigma \nabla v) \quad\mbox{ in }\  \mathbb R^N\times (0,+\infty) \ \mbox{ and }\ v\ =  \mu{\mathcal X}_{\mathcal N}\ \mbox{ on } \mathbb R^N\times \{0\}.
$$
The number $\mu > 0$ can be chosen such that
$$
v \ge 1\ (\ge u) \mbox{ on } \partial\Omega \times (0,1],
$$
because \eqref{Gaussian bounds} implies that 
$$
  v(x,t) \ge \mu M^{-1} t^{-\frac N2}\int\limits_{\mathcal N}e^{-\frac{\alpha|x-\xi|^2}{t}} d\xi \ge \mu M^{-1} t^{-\frac N2}\int\limits_{B_{\rho/2}(0)}e^{-\frac{\alpha|\xi|^2}{t}} d\xi 
$$
for $(x,t) \in \partial\Omega\times(0,+\infty)$.
Thus, the comparison principle yields that
\begin{equation}
\label{upper bound by v}
  u \le v\ \mbox{ in } \Omega\times (0,1].
\end{equation}
On the other hand, it follows from \eqref{Gaussian bounds} that
$$
v(x,t) \le \mu M t^{-\frac N2}\int\limits_{\mathcal N}e^{-\frac{|x-\xi|^2}{\alpha t}} d\xi\quad \mbox{ for } (x,t) \in \mathbb R^N\times(0,+\infty)
$$
and hence, since $|x-\xi| \ge \rho/2$ for every $x \in \Omega_\rho$ and $\xi \in \mathcal N$, we obtain that
$$
 v(x,t) \le \mu M t^{-\frac N2}e^{-\frac {\rho^2}{8\alpha t}}\int\limits_{\mathbb R^N}e^{-\frac{|x-\xi|^2}{2\alpha t}} d\xi = \mu M (2\pi\alpha)^{\frac N2}e^{-\frac {\rho^2}{8\alpha t}}
 $$
for every $(x,t) \in \Omega_\rho \times (0,+\infty)$. 
\par
This inequality and 
\eqref{upper bound by v} then yield the first formula of (2).   
\end{proof}
 
\medskip


Next lemma informs us that, as in the case of stationary level surfaces, surfaces having the constant flow property satisfy a certain balance law.

\begin{lemma}[A balance law] 
\label{le: balance law}
Let $\Gamma$ be a connected component of class $C^2$ of $\partial G$ satisfying \eqref{nearest component}. 
Set $r_0 = \mbox{\rm dist}(\Gamma,\partial\Omega) ( > 0).$ 
\par
Let $u$ be the solution of either problem \eqref{heat equation initial-boundary}--\eqref{heat initial} or \eqref{heat Cauchy}. 
Then, \eqref{constant flow surface partially} holds if and only if there exists
a function $c:(0, r_0)\times(0,+\infty)\to\RR$ such that 
\begin{equation}
\label{balance law}
\int_{B_r(x)} u(y,t)\,(y-x)\cdot\nu(x)\,dy = c(r,t) \
\mbox{ for every }\  (x, r,t) \in \Ga\times(0, r_0)\times(0,+\infty),
\end{equation}
where $\nu=\nu(x)$ denotes the outward unit normal vector to $\Ga$ at $x\in\Ga$.
\end{lemma}

\begin{proof}
Since $\Gamma$ is compact, let $p\in\Ga$ be a point such that $\dist(p,\pa\Om)=r_0$.
If \eqref{constant flow surface partially} holds, we have that
\begin{equation}
\label{constant flow property 1}
d(t) = \sigma_s\nabla u(p, t)\cdot\nu(p) = \sigma_s\nabla u(q, t)\cdot\nu(q)\ \mbox{ for every } (q,t) \in \Gamma \times (0,+\infty).
\end{equation}
\par
Next, fix a $q\in \Gamma$ and let $A$ be an orthogonal matrix satisfying
\begin{equation}
\label{by rotation}
A\nu(p) = \nu(q).
\end{equation}
From  \eqref{constant flow property 1} and  \eqref{by rotation}
we obtain that  the function $v = v(x,t)$, defined by
$$
v(x,t) = u(x+p, t) - u(Ax + q, t)\ \mbox{ for } (x,t )\in B_{r_0}(0) \times (0,+\infty),
$$
is such that
\begin{multline*}
\nabla v(0,t)\cdot\nu(p) = \nabla u(p,t)\cdot\nu(p)- [A^{T}\nabla u(q,t)]\cdot\nu(p) = \\
\nabla u(p,t)\cdot\nu(p)- \nabla u(q,t)\cdot[A\,\nu(p)]=
\nabla u(p,t)\cdot\nu(p)- \nabla u(q,t)\cdot\nu(q)=0,
\end{multline*}
for every $t > 0$. Here, the superscript $T$ stands for transpose.
\par
Now, since assumption \eqref{near the boundary} guarantees that $B_{r_0}(p)$ and $B_{r_0}(q)\subset\Omega\setminus\overline{D}$, and  $\sigma = \sigma_s$ in $\Omega\setminus\overline{D}$, we have that $v$ satisfies the heat equation with constant conductivity $\sigma_s$:
\begin{equation*}
v_t =\sigma_s \Delta v\ \mbox{ in }\ B_{r_0}(0)\times (0,+\infty).
\end{equation*}
Thus, also the function $\nabla v(x,t)\cdot\nu(p)$ satisfies the same equation and we have seen that 
$\nabla v(0,t)\cdot\nu(p) =0$ for every $t>0$. Hence, 
we can use a balance law (see \cite[Theorem 2.1, pp. 934--935]{MSannals2002} or \cite[Theorem 4, p. 704]{MSmathz1999}) to obtain that
\begin{equation*}
\label{balance law-v}
\int\limits_{\partial B_r(0)}\!\!\nabla v(y,t)\cdot\nu(p)\, dS_y = 0\ \mbox{ for every }\  (r,t) \in (0, r_0)\times(0,+\infty)
\end{equation*}
or, by integrating this in $r$, that
$$
\int\limits_{B_r(0)}\!\!\nabla v(y,t)\cdot\nu(p)\,dy = 0\ \mbox{ for every }\  (r,t) \in (0, r_0)\times(0,+\infty).
$$
By the divergence theorem and again integrating in $r$, we then get
$$
\int\limits_{B_{r}(0)}\!\!v(y,t)\,y\cdot\nu(p)\, dy = 0\ \mbox{ for every }\  (r,t) \in (0, r_0)\times(0,+\infty),
$$
that is 
\begin{eqnarray}
&&\int\limits_{B_{r}(p)}\!\!u(y,t)(y-p)\cdot\nu(p)\, dy = \int\limits_{B_{r}(q)}\!\!u(y,t)(y-q)\cdot\nu(q)\, dy \label{balance law special}
\\ 
&&\qquad\qquad\mbox{ for every }\  (q,r,t) \in \Gamma\times(0,r_0)\times(0,+\infty)\nonumber.
\end{eqnarray}
Therefore, \eqref{balance law} ensues. 
\par
It is not difficult to show that \eqref{balance law} implies \eqref{constant flow surface partially}.
\end{proof}

\medskip

The following lemma is decisive to prove Theorem \ref{th:constant flow}. Among other things, it states that, as in the case of stationary isothermic surfaces, also surfaces having the constant flow property are parallel to a connected component of  $\pa\Om$.

\begin{lemma} 
\label{le: constant weingarten curvature}
Let $u$ be the solution of either problem \eqref{heat equation initial-boundary}--\eqref{heat initial} or \eqref{heat Cauchy},  and let
$\Gamma$ be a connected component of class $C^2$ of $\partial G$ satisfying \eqref{nearest component}. 
Under the assumption \eqref{constant flow surface partially} of {\rm Theorem \ref{th:constant flow}}, the following assertions hold:
\begin{enumerate}[\rm (1)]
\item there exists a number $r_0 > 0$ such that 
$$
\mbox{\rm dist}(x, \partial\Omega) = r_0\ \mbox{ for every } x \in \Gamma;
$$
\item $\Gamma$ is a real analytic hypersurface;
\item there exists a connected component $\gamma$ of $\partial\Omega$, that is also a real analytic hypersurface, such that the mapping $\gamma \ni y \mapsto x(y)  \equiv y-r_0\,\nu(y) \in \Gamma$ is a diffeomorphism; in particular $\gamma$ and $\Gamma$  are parallel hypersurfaces at distance $r_0$;
\item it holds that
\begin{equation*}
\label{bounds of curvatures}
 \kappa_j(y) < \frac 1{r_0}\ \mbox{ for every } y \in \gamma \mbox{ and } \ j=1,\dots, N-1;
\end{equation*}
\item there exists a number $c_0 > 0$ such that $\Pi_{\pa\Om}(r_0, y)= c_0$ for every $y\in\gamma$,
where $\Pi_{\pa\Om}$ is given in \eqref{product-curvatures}.
\end{enumerate}
\end{lemma}
 
\begin{proof}
We just have to prove assertion (1): the remaining ones then will easily follow. 
\par
Let $r_0>0$ be the minimum of $\dist(x,\pa\Om)$ for $x\in\Ga$ and suppose it is achieved at $p$; assume that there exists a point $q_* \in \Gamma$ such that
$$
r_0 < \dist(q_*,\partial\Omega).
$$
Since $\overline{B_{r_0}(q_*)} \subset \Omega$, with the aid of  Lemma \ref{le:initial behavior and decay at infinity}  we have:
\begin{equation}
\label{decay to zero at q*}
\lim_{t \to +0} t^{-\frac {N+1}4}\!\!\!\int\limits_{B_{r_0}(q_*)}\!\!u(x,t)\,(x-q_*)\cdot\nu(q_*) \, dx=0.
\end{equation}
In view of \eqref{nearest component}, since $r_0 = \mbox{ dist}(p,\partial\Omega) = \mbox{ dist}(\partial G,\partial\Omega) = \mbox{ dist}(\overline{G},\partial\Omega)$ and $\Gamma$ is of class $C^2$, we can find a ball
$B_\delta(z) \subset G$ satisfying
$$
\overline{B_\delta(z)}\cap\partial G = \{p\}\ \mbox{ and }\ B_{\delta+r_0}(z) \subset \Omega.
$$
Also, by setting $\hat{p} = p+r_0\nu(p)\ (\in \partial\Omega)$ we have:
\begin{equation}
\label{curvature estimates at p-hat}
\overline{B_{r_0}(p)}\cap\partial\Omega = \{\hat{p}\}\ \mbox{ and }\ \kappa_j(\hat{p}) \le \frac 1{r_0+\delta} < \frac 1{r_0}\ \mbox{ for } j=1, \dots, N-1.
\end{equation}
\begin{figure}[h]
\centering
\includegraphics[width=0.55\textwidth]{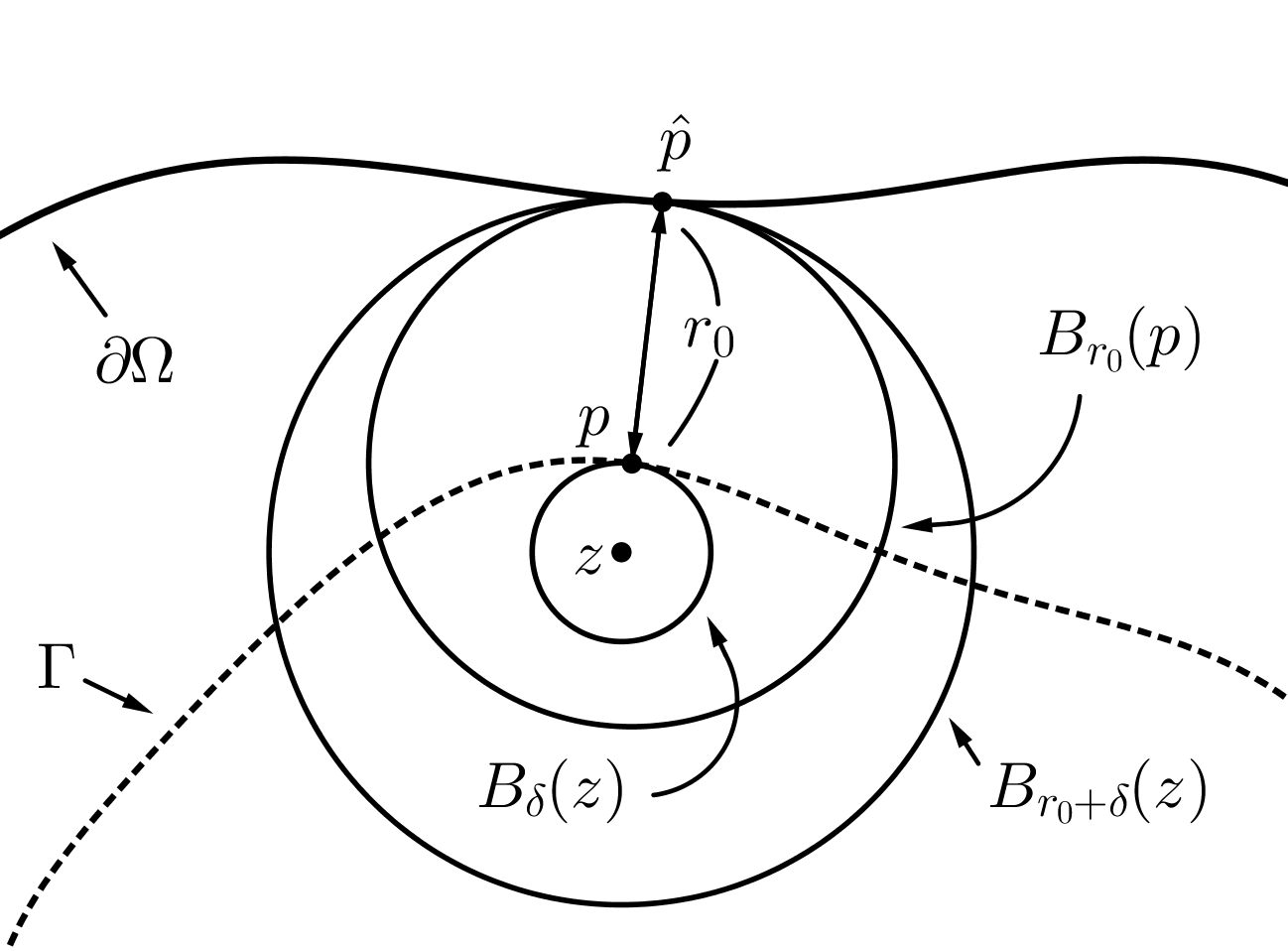}
\caption{The three-balls construction.}
\label{picture1}
\end{figure}

\vskip 1ex
\noindent
Thus, Proposition \ref{prop:heat content asymptotics} gives that
\begin{equation}
\label{asymptotics and curvatures at p-hat}
\lim_{t\to +0}t^{-\frac{N+1}4 }\!\!\!\int\limits_{B_{r_0}(p)}\!\! u(x,t)\ dx=
\frac{C(N, \sigma)}{\sqrt{\Pi_{\pa\Om}(r_0, \hat{p})}} .
\end{equation}
On the other hand, by Lemma \ref{le:initial behavior and decay at infinity}  and the fact that
$\overline{B_{r_0}(p)}\cap\partial\Omega = \{\hat{p}\}$, we have
\begin{equation}
\label{decay to zero outside the neighborhood of p-hat}
\lim_{t \to +0} t^{-\frac {N+1}4}\!\!\!\!\!\!\!\!\!\!\int\limits_{B_{r_0}(p)\setminus B_\varepsilon(\hat{p})}\!\!\!\!\!\! u(x,t)\ dx = 0\ \mbox{ for every } \varepsilon >0.
\end{equation}
Therefore, combining the last two formulas yields that
\begin{equation}
\label{relationship between the limit and principal curvatures}
\lim_{t \to +0} t^{-\frac {N+1}4}\!\!\!\int\limits_{B_{r_0}(p)}\!\!u(x,t)(x-p)\cdot\nu(p)\,dx = r_0\,\frac{C(N, \sigma)}{\sqrt{\Pi_{\pa\Om}(r_0, \hat{p})}}.
\end{equation}
In fact,  for every $\varepsilon > 0$, we have
\begin{eqnarray}
&&\left|  t^{-\frac {N+1}4}\!\!\!\!\!\!\!\!\!\int\limits_{B_{r_0}(p)\cap B_\varepsilon(\hat{p})}\!\!\!\!\!\!\!\!u(x,t)(x-\hat{p})\cdot\nu(p)\,dx\right| 
\le \varepsilon t^{-\frac {N+1}4}\!\!\!\int\limits_{B_{r_0}(p)}\!\!u(x,t)\,dx, \label{small near the touching point}
\\
&&\left|  t^{-\frac {N+1}4}\!\!\!\!\!\!\!\!\!\int\limits_{B_{r_0}(p)\setminus B_\varepsilon(\hat{p})}\!\!\!\!\!\!\!\!u(x,t)(x-\hat{p})\cdot\nu(p)\,dx\right| 
\le 2r_0 t^{-\frac {N+1}4}\!\!\!\!\!\!\!\!\!\int\limits_{B_{r_0}(p)\setminus B_\varepsilon(\hat{p})}\!\!\!\!\!\!\!\!u(x,t)\,dx.\label{small far from the touching point}
\end{eqnarray}
Moreover, since $(\hat{p}-p)\cdot\nu(p)=r_0$, we have that 
\begin{multline*}
t^{-\frac {N+1}4}\!\!\!\!\!\!\int\limits_{B_{r_0}(p)}\!\!u(x,t)(x-p)\cdot\nu(p)\,dx= 
r_0\, t^{-\frac {N+1}4}\!\!\!\!\!\!\int\limits_{B_{r_0}(p)}\!\!u(x,t)\,dx+  \\
 t^{-\frac {N+1}4}\!\!\!\!\!\!\!\!\!\!\!\!\int\limits_{B_{r_0}(p)\setminus B_\varepsilon(\hat{p})}\!\!\!\!\!\!\!\!u(x,t)(x-\hat{p})\cdot\nu(p)\,dx +
  t^{-\frac {N+1}4}\!\!\!\!\!\!\!\!\!\!\!\!\int\limits_{B_{r_0}(p)\cap B_\varepsilon(\hat{p})}\!\!\!\!\!\!\!\!u(x,t)(x-\hat{p})\cdot\nu(p)\,dx,
 \end{multline*}
 for every $t > 0$.  Therefore, combining \eqref{asymptotics and curvatures at p-hat},  \eqref{decay to zero outside the neighborhood of p-hat}, \eqref{small near the touching point} and \eqref{small far from the touching point} yields that 
  \begin{eqnarray*}
 &&(r_0 -\varepsilon)\,\frac{C(N, \sigma)}{\sqrt{\Pi_{\pa\Om}(r_0, \hat{p})}} \le \liminf_{t \to +0}\  t^{-\frac {N+1}4}\!\!\!\int\limits_{B_{r_0}(p)}\!\!u(x,t)(x-p)\cdot\nu(p)\,dx\\
  && \le \limsup_{t \to +0}\  t^{-\frac {N+1}4}\!\!\!\int\limits_{B_{r_0}(p)}\!\!u(x,t)(x-p)\cdot\nu(p)\,dx \le (r_0 +\varepsilon)\,\frac{C(N, \sigma)}{\sqrt{\Pi_{\pa\Om}(r_0, \hat{p})}}
 \end{eqnarray*}
 for every $\varepsilon > 0$,  which gives \eqref{relationship between the limit and principal curvatures}.
\par
It is clear that \eqref{relationship between the limit and principal curvatures} contradicts \eqref{decay to zero at q*} and the balance law \eqref{balance law}, and hence assertion (1) holds true.
\par
Now, once we have (1), we can apply the same argument as above to any other point in $\Gamma$. Thus, we know from \eqref{balance law}, \eqref{curvature estimates at p-hat} and \eqref{relationship between the limit and principal curvatures} that there exists a connected component $\gamma$ of $\partial\Omega$ satisfying
(3), (4) and (5). The analyticity of $\gamma$ follows from (4) and (5). Indeed, by using local coordinates, the condition (5) with (4) can be converted into a second order analytic nonlinear elliptic equation of Monge-Amp\`ere type, where (4) guarantees  the ellipticity as is noted in \cite[p. 945]{MSannals2002}.  Hence (2) is implied by (3) together with (4). 
\end{proof}


\subsection{Proof of Theorem \ref{th:constant flow} for problem \eqref{heat equation initial-boundary}--\eqref{heat initial}}
\label{subsec:initial-boundary}
Let $u$ be the solution of  problem \eqref{heat equation initial-boundary}--\eqref{heat initial}. 
By virtue of (1) of Lemma \ref{le:initial behavior and decay at infinity},  we can define the function $v:\ol{\Om}\to\RR$ by the Laplace transform of $1-u(x, \cdot)$ computed at the complex parameter $1 + 0 \sqrt{-1}$
\begin{equation}
\label{def-v}
v(x)=\int_0^\infty e^{-t} [1-u(x,t)]\,dt \quad \mbox{ for } x \in \overline{\Omega},
\end{equation}
and set $U=v$ on $\overline{\Om}\setminus D$ and $V=v$ on $\overline{D}$. 
Then, it is easy to show that
\begin{eqnarray}
&&0< U < 1\mbox{ in } \Omega \setminus\overline{D},\quad 0< V < 1 \mbox{ in } D,\label{all bounded from above and below 1}
\\
&&\sigma_s\,\Delta U =U - 1  \mbox{ in } \Omega \setminus\overline{D},\quad \sigma_c\,\Delta V =V- 1 \mbox{ in } D, \label{poisson and Laplace equations 1}
\\
&&U = V\ \mbox{ and }\  \sigma_s\,\pa_\nu U= \sigma_c\, \pa_\nu V  \ \mbox{ on } \partial D, \label{transmission condition between U and V 1}
\\
&&U = 0\  \mbox{ on } \partial\Omega. \label{homogenious Dirichlet condition 1}
\end{eqnarray}
Here, $\nu$ denotes the outward unit normal vector to $\partial D$ at points of $\partial D$. The two equations in \eqref{transmission condition between U and V 1} follow from the transmission condition satisfied by $u$ on $\partial D \times (0,+\infty)$ and involve the continuous extensions of the relevant functions up to $\pa D$.
\par
Next, let $\ga$ be the connected component of $\pa\Om$ whose existence is guaranteed by Lemma \ref{le: constant weingarten curvature}. Claims (5) and (4) of Lemma \ref{le: constant weingarten curvature} also tell us that $\ga$ is an elliptic Weingarten-type surface, that is its principal curvatures satisfy a symmetric constraint which can be recast as an elliptic partial differential equation,  considered by Aleksandrov's sphere theorem \cite[p. 412]{Alek1958vestnik}, and hence $\gamma$ is a sphere. Consequently, $\Ga$ is a sphere concentric with $\ga$;
we can always assume that the origin is their common center. 
\par
By combining the initial and boundary conditions of problem \eqref{heat equation initial-boundary}--\eqref{heat initial} and  the assumption \eqref{constant flow surface partially} with the real analyticity in $x$ of $u$ over $\Omega\setminus\overline{D}$, we see that $u$ is radially symmetric in $x$ on $\overline{\Omega}\setminus D$ for every $t>0$. Here, we used the fact that $\Omega\setminus\overline{D}$ is connected. Moreover, in view of \eqref{heat Dirichlet}, we can distinguish two cases:
$$
\mbox{\rm (I)  } \Omega \mbox{ is a ball;}\qquad \mbox{\rm (II)  } \Omega \mbox{ is a spherical shell.}
$$

\vskip 2ex

 We first show that case (II) never occurs.  Suppose that  
$
\Omega = B_{\rho_+} \setminus \overline{B_{\rho_-}} 
$
where $B_{\rho_+}$ and  $B_{\rho_-}$ are two balls centered at the origin with $\rho_+ > \rho_- > 0$.
By the radial symmetry of $u$ on $\Omega\setminus\overline{D}$ for every $t>0$, being $\Omega \setminus \overline{D}$ connected, there exists a function $\widetilde{U}:[\rho_-,\rho_+]\to\RR$  such that $U(x) = \widetilde{U}(|x|)$ for $x \in \overline{\Omega}\setminus D$. Moreover,
 by \eqref{poisson and Laplace equations 1}, $\widetilde{U}$ is extended as a solution of 
 \begin{equation*}
 \label{ODE for U}
 \sigma_s\left(\pa_{rr}\widetilde{U} + \frac {N-1}r\pa_r\widetilde{U}\right)= \widetilde{U}-1\ \mbox{ for all } r > 0,
 \end{equation*}
  where $\pa_r$ and $\pa_{rr}$ stand for first and second derivatives with respect to the variable $r=|x|$.
That means that $U$ is extended as a radially symmetric solution of $\sigma_s\Delta U =U - 1$  in 
$\mathbb R^N\setminus\{0\}$.
 By applying Hopf's boundary point lemma (see \cite[Lemma 3.4, p. 34]{GT1983}) to $U$, we obtain from 
 \eqref{all bounded from above and below 1},  \eqref{poisson and Laplace equations 1}  and \eqref{homogenious Dirichlet condition 1}
that 
\begin{eqnarray}
&& \sigma_s\Delta U= U-1 < 0 \  \mbox{ in } \Omega, \label{extended ODE for U 1}
\\
&& \pa_\nu U=-\pa_r\widetilde{U}(\rho_-) <0 \ \mbox{ on } \pa B_{\rho_-} \ \mbox{ and }  \ \pa_\nu U=\pa_r\widetilde{U}(\rho_+)  < 0 \ \mbox{ on } \pa B_{\rho_+}. \label{sign of derivatives of U on the boundary of Omega 1}
\end{eqnarray}
\par
Now, we use Lemma \ref{le: the unique determination}.
We set $D_1=\varnothing$, $D_2 = D$, and consider two functions $v_j\in H^1(\Omega)\ (j=1,2)$ defined by 
\begin{equation*}
\label{two functions in Lemma 2.5}
v_1 = U \ \mbox{ and } v_2 = \left\{\begin{array}{rll}
 U  \ &\mbox{ in }\ \Omega \setminus D,
\\
 V \ &\mbox{ in }\  D.
\end{array}\right.
\end{equation*}
In view of \eqref{poisson and Laplace equations 1},  \eqref{transmission condition between U and V 1}, \eqref{homogenious Dirichlet condition 1}, \eqref{extended ODE for U 1} and \eqref{sign of derivatives of U on the boundary of Omega 1}, Lemma \ref{le: the unique determination} gives that $v_1=v_2$ in $\Omega$  and $\varnothing=D$, which is a contradiction. 
Thus, case (II) never occurs. 

\vskip 2ex
It remains to consider case (I), that is we assume that $\Omega$ is a ball $B_R$ centered at the origin
for some radius $R > 0$. 
\par
Since $u$ is radially symmetric on $\overline{\Omega}\setminus D$ for every $t>0$ and $\Omega \setminus \overline{D}$ is connected, by applying Hopf's boundary point lemma to the radially symmetric function $U$, we obtain from 
 \eqref{all bounded from above and below 1},  \eqref{poisson and Laplace equations 1}  and \eqref{homogenious Dirichlet condition 1}
that 
\begin{equation}
\label{sign of derivatives of U on the boundary of Omega 2}
 \sigma_s\,\pa_\nu  U= \sigma_s\, \pa_r\widetilde{U}(R)< 0 \ \mbox{ on } \partial B_R.
 \end{equation}
 \par
Thus, in view of \eqref{all bounded from above and below 1},  \eqref{poisson and Laplace equations 1}  and \eqref{homogenious Dirichlet condition 1}, we see that the function $v$ defined in \eqref{def-v} satisfies
$$
\mbox{\rm div}(\sigma\nabla v) = v -1 < 0\ \mbox{ in } B_R\ \mbox{ and }\  v = 0 \ \mbox{ on } \partial B_R.
$$
Therefore, with the aid of \eqref{sign of derivatives of U on the boundary of Omega 2}, we can apply Theorem \ref{th:constant Neumann boundary condition} to $v$ to see that $D$ must be a ball centered at the origin.


\subsection{Proof of Theorem \ref{th:constant flow} for problem  \eqref{heat Cauchy}}
\label{subsection 3.3}

Let $u$ be the solution of  problem  \eqref{heat Cauchy}. We proceed similarly to Subsection \ref{subsec:initial-boundary}. This time, by virtue of (1) of Lemma \ref{le:initial behavior and decay at infinity}, we define a function $v:\RR^N\to\RR$ by
\begin{equation}
\label{def-v-cauchy}
v(x)=\int_0^\infty e^{-t} [1-u(x,t)]\,dt \quad \mbox{ for every } \ x\in\RR^N
\end{equation} 
and, in addition to the already defined functions $U$ and $V$, we set $W=v$ on $\RR^N\setminus\ol{\Om}$.
\par
While $U$ and $V$ satisfy \eqref{all bounded from above and below 1}-\eqref{transmission condition between U and V 1}, $W$ satisfies
\begin{eqnarray}
&&0<W<1 \ \mbox{ in } \ \RR^N\setminus\ol{\Om},\label{W bounded from above and below} 
\\
&&\si_m\,\De W=W \ \mbox{ in } \ \RR^N\setminus\ol{\Om}, \label{poisson for W}
\\
&&W= U\ \mbox{ and }\ \sigma_m\,\pa_\nu W=\sigma_s\, \pa_\nu U  \ \mbox{ on } \partial\Omega, \label{transmission condition between U and W 2}
\\
&& \lim_{|x| \to \infty} W(x) = 0. \label{decay at infinity}
\end{eqnarray}
Similarly to Subsection \ref{subsec:initial-boundary}, $\nu$ denotes the outward unit normal vector to $\partial D$ or to $\partial\Omega$, and  both \eqref{transmission condition between U and V 1}  and   \eqref{transmission condition between U and W 2} are consequences of the transmission conditions satisfied by $u$ on $\partial D \times (0,+\infty)$ and on $\partial \Omega \times (0,+\infty)$, respectively.  Also,  to obtain \eqref{decay at infinity}, we used Lemma \ref{le:initial behavior and decay at infinity}  together with Lebesgue's dominated convergence theorem.

Again, by Aleksandrov's sphere theorem \cite[p. 412]{Alek1958vestnik}, Lemma \ref{le: constant weingarten curvature} yields that $\gamma$ and $\Gamma$ are concentric spheres, with a common center that we can place at the origin.  Being $\Omega\setminus\overline{D}$ connected, the radial symmetry of $u$ in $x$ on $\overline{\Omega}\setminus D$ for every $t>0$ is obtained similarly, by combining the initial  condition in \eqref{heat Cauchy}  and  the assumption \eqref{constant flow surface partially} with the real analyticity in $x$ of $u$ over $\Omega \setminus\overline{D}$. 
\par
Moreover, in view of the initial  condition of problem  \eqref{heat Cauchy} and Proposition \ref{prop:the initial limits on the interface}, we can prove that $\Omega$ is radially symmetric and hence $u$ is radially symmetric in $x$ on $\mathbb R^N\setminus D$ for every $t>0$. Indeed, if there exists another connected component $\hat{\gamma}$ of $\partial\Omega$, which is not a sphere centered at the origin, 
we can find a number $\rho > 0$ and two points $p \in \partial\Omega, q \in \Omega\setminus\overline{D}$ such that
$$
\partial B_\rho \subset \overline{\Omega},\ p \in \hat{\gamma}\cap\partial B_\rho,\ \mbox{ and } q \in (\Omega\setminus\overline{D})\cap\partial B_\rho,
$$
being $B_\rho$ the ball centered at the origin with radius $\rho$.
\begin{figure}[h]
\centering
\includegraphics[width=0.55\textwidth]{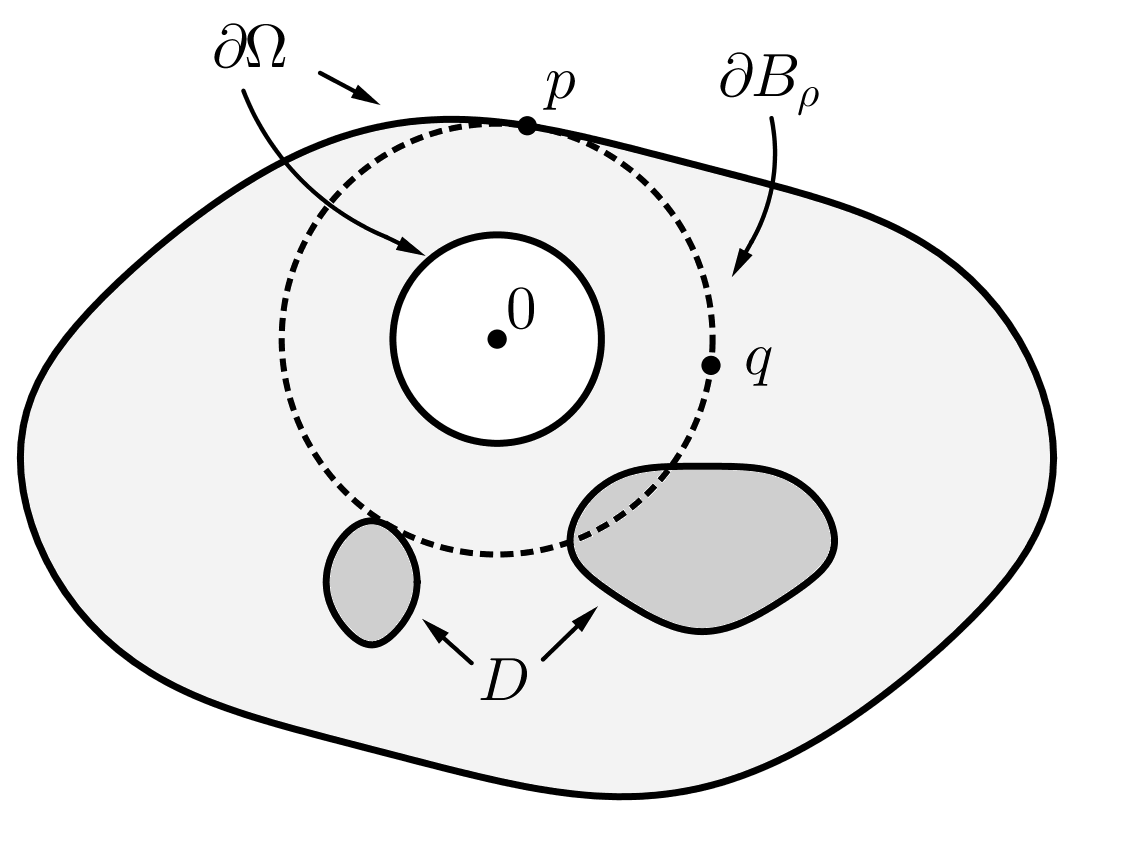}
\caption{The ball construction for the Cauchy problem.}
\label{picture2}
\end{figure}

Then, since $u$ is radially symmetric on $\overline{\Omega}\setminus D$ for every $t>0$, we have:
\begin{equation}
\label{equality from radial symmetry}
u(p,t) = u(q,t)\ \mbox{ for every } t>0.
\end{equation}
On the other hand, by Proposition \ref{prop:the initial limits on the interface} $\lim\limits_{t \to +0}u(p,t) = \frac {\sqrt{\sigma_m}}{\sqrt{\sigma_s}+\sqrt{\sigma_m}}$ and by (2) of Lemma \ref{le:initial behavior and decay at infinity}  $\lim\limits_{t \to +0}u(q,t) = 0$. These contradict \eqref{equality from radial symmetry}. 
Once we know that $\Omega$ is radially symmetric, the radial symmetry of $u$ on $\mathbb R^N\setminus D$ for every $t>0$ follows from the initial  condition in \eqref{heat Cauchy}.

\medskip

Thus, as in the previous case, we can distinguish two cases:
$$
\mbox{\rm (I)  } \Omega \mbox{ is a ball;}\qquad \mbox{\rm (II)  } \Omega \mbox{ is a spherical shell.}
$$

We first show that case (II) never occurs.  With the same notations as in Subsection \ref{subsec:initial-boundary}, we set 
 $\Omega = B_{\rho_+} \setminus \overline{B_{\rho_-}}$.
 Since $u$ is radially symmetric in $x$ on $\mathbb R^N\setminus D$ for every $t>0$, so is $W$ on $\mathbb R^N\setminus D$. 
 Observe from \eqref{W bounded from above and below} and \eqref{poisson for W}  that 
 $$
 \Delta W > 0\   \mbox{ in }  B_{\rho_-}  \mbox{ and } \mathbb R^N\setminus \overline{B_{\rho_+}}.
 $$
Therefore,  in view of \eqref{decay at infinity}, the strong maximum principle tells us that the positive maximum value of $W$ on $\ol{B_{\rho_-}}$ or on $\mathbb R^N \setminus B_{\rho_+}$ is achieved only on $\partial B_{\rho_-}$ or $\partial B_{\rho_+}$, respectively.  Hence,  since $W$ is radially symmetric, Hopf's boundary point lemma yields that
 \begin{equation}
 \label{signs of radial derivatives}
 \pa_\nu W< 0 \mbox{ on  } \pa B_{\rho_-} \mbox{ and }  \partial B_{\rho_+}.
 \end{equation}
  
As in Subsection \ref{subsec:initial-boundary},  $U$ is extended as a radially symmetric solution of $\sigma_s\,\Delta U =U - 1$  in $ \mathbb R^N\setminus\{0\}$.  Then, it follows from \eqref{signs of radial derivatives},  \eqref{all bounded from above and below 1} and \eqref{transmission condition between U and W 2} that both \eqref{extended ODE for U 1} and \eqref{sign of derivatives of U on the boundary of Omega 1} also hold true. Therefore, with the aid of Lemma \ref{le: the unique determination}, by the same argument of the proof in Subsection \ref{subsec:initial-boundary}, we obtain a contradiction, and hence  case (II) never occurs.

\medskip

It remains to consider case (I). As in Subsection \ref{subsec:initial-boundary}, we set 
$
\Omega = B_{R}.
$
Since $u$ is radially symmetric in $x$ on $\mathbb R^N\setminus\overline{D}$ for every $t>0$,  $W$ is also radially symmetric on $\mathbb R^N\setminus\overline{D}$. 
 Observe from \eqref{W bounded from above and below}  and \eqref{poisson for W}  that 
 $$
 \Delta W > 0\   \mbox{ in }  \mathbb R^N\setminus \ol{B_R}.
 $$
Therefore,  in view of \eqref{decay at infinity},  the strong maximum principle informs us that the positive maximum value of $W$ on $\mathbb R^N \setminus B_R$ is achieved  only on $\partial B_{R}$.  Hence,  since $W$ is radially symmetric, Hopf's boundary point lemma yields that
 \begin{equation}
 \label{sign of radial derivative}
 \pa_\nu W < 0 \ \mbox{ on } \ \pa B_R.
 \end{equation}
 Combining  \eqref{sign of radial derivative}  with \eqref{transmission condition between U and W 2} implies that both $U$ and $\pa_\nu U$ are constant on $\pa B_R$.
 Therefore, with the aid of  Theorem \ref{th:constant Neumann boundary condition} and by the same argument of the proof in Subsection \ref{subsec:initial-boundary}, we conclude that $D$ must be a ball centered at the origin.

\setcounter{equation}{0}

\subsection{Proof of Theorem \ref{th:stationary isothermic surface}}
\label{subsection 3.4}

In view of the statements of Theorems \ref{th:stationary isothermic surface}, \ref{th:stationary isothermic} and \ref{th:stationary isothermic cauchy}, it suffices to show that Theorem \ref{th:stationary isothermic cauchy} can be improved as in Theorem \ref{th:stationary isothermic}. Namely, in proposition (b) of Theorem \ref{th:stationary isothermic cauchy} we may show that the assumption that
$\sg_s=\sg_m$ is not necessary.
\par
Let in fact $u$ be the solution of  problem  \eqref{heat Cauchy}.  Aleksandrov's sphere theorem \cite[p. 412]{Alek1958vestnik} and  \cite[Lemma 2.4, p. 176]{Strieste2016} yield that $\gamma$ and $\Gamma$ are concentric spheres. Then,  with the aid of the initial  condition of problem  \eqref{heat Cauchy} and Proposition \ref{prop:the initial limits on the interface}, we can observe that the rest of the proof runs as in the proof given in Subsection \ref{subsection 3.3}.

\setcounter{equation}{0}

\section{The constant flow property at the boundary}
\label{section4}

In this section, we will give the proofs of Theorems \ref{th:constant flow serrin} and \ref{th:stationary isothermic three-phase} .

Let $u$ be the solution of  problem \eqref{heat equation initial-boundary}--\eqref{heat initial}, and let 
$\Gamma$ be a connected component of $\partial\Omega$.
We introduce the distance function $\delta = \delta(x)$ of $x \in \mathbb R^N$  to $\Gamma$ by
\begin{equation}
\label{distance function to the boundary of the domain}
\delta(x) = \mbox{ dist}(x,\Gamma)\ \mbox{ for }\ x \in \mathbb R^N.
\end{equation}
Since $\Gamma$ is of class $C^6$ and compact,  by choosing a number $\delta_0 > 0$ sufficiently small and setting  
\begin{equation}\label{inner tubular neighborhood of Omega}
\mathcal N_0 = \{ x \in \Omega\ :\ 0< \delta(x) < \delta_0 \},
\end{equation}
we see that
\begin{eqnarray}
&& \overline{\mathcal N_0} \cap \overline{D} = \varnothing,\  \delta \in C^6(\overline{\mathcal N_0}), \label{c6 regularity}
\\
&&\mbox{ for every } x \in \overline{\mathcal N_0} \mbox{ there exists a unique }y = y(x) \in\Gamma \mbox{ with } \delta(x) = |x-y|, \label{the nearest point y from x}
\\
&& y(x) = x -\delta(x)\nabla\delta(x)\ \mbox{ for all } x \in \overline{\mathcal N_0}, \label{ the point y and distance from x} 
\\
&&
 \max_{1\le j \le N-1}\kappa_j(y) < \frac 1{2\delta_0}\ \mbox{ for every } y \in \Gamma. \label{upper bound of the curvatures on Gamma}
\end{eqnarray}
The principal curvatures $\ka_j$ of $\Ga$ are taken at $y$ with respect to the inward unit normal vector $-\nu(y)=\nabla\delta(y)$ to $\partial\Omega$.  

\subsection{Introducing a Laplace transform}
\label{subsec:laplace}
Let us define the function $w = w(x, \lambda)$ by the Laplace-Stieltjes transform of $u(x, \cdot)$ or the Laplace transform of $u_t(x,\cdot)$ restricted on the semiaxis of real positive numbers
$$
w(x,\lambda) = \lambda \int_0^\infty e^{-\lambda t}u(x,t)\ dt\ \mbox{ for } (x,\lambda) \in \Omega \times (0,+\infty).
$$
Notice that letting $\lambda = 1$ gives
\begin{equation}
\label{relationship with the auxiliary function in section 3}
w(x,1) =1-v(x)\ \mbox{ for every } x \in \Omega,  \mbox{ and } w(x,1)=1-U(x)  \mbox{ for } x\in\Om\setminus D,
\end{equation}
where $v$ is the function defined by \eqref{def-v} and $U = \restr{v}{\overline{\Omega}\setminus D}$.
\par
Next, we observe that for every $\lambda > 0$
\begin{eqnarray}
& \mbox{ div}(\sigma \nabla w) - \lambda w = 0\ \mbox{ and }\ 0 < w < 1 &\mbox{ in } \Omega, \label{elliptic pde for lambda}
\\
& w = 1 \ &\mbox{ on } \partial\Omega. \label{boundary condition for lambda}
\end{eqnarray}
Hence, by the assumption \eqref{constant flow surface partially},  there exists a function  $d_0: (0,\infty)\to \RR$ satisfying 
\begin{equation}
\label{constant Neumann serrin type}
\sigma_s\,\pa_\nu w(x,\lambda) = d_0(\lambda) \ \mbox{ for every } (x,\lambda) \in \Gamma\times(0,+\infty).
\end{equation}
Moreover, it follows from the first formula of (2) of Lemma \ref{le:initial behavior and decay at infinity}  that there exist two positive constants $\widetilde{B}$ and $\widetilde{b}$ satisfying
\begin{equation}
\label{exponential decay for elliptic eq}
0 < w(x,\lambda) \le \widetilde{B}e^{-\widetilde{b}\sqrt{\lambda}}\ \mbox{ for every } (x,\lambda) \in \left(\partial \mathcal N_0 \cap\Omega\right)\times(0,+\infty).
\end{equation}

\subsection{Two auxiliary functions}\label{4.2}

Since $w$ satisfies \eqref{boundary condition for lambda} and $\Delta w - \frac \lambda{\sigma_s}w=0$ in $\mathcal N_0$,   in view of the formal WKB approximation of $w$ for  sufficiently large $\tau = \frac \lambda{\sigma_s}$
$$
w(x,\lambda) \sim e^{-\sqrt{\tau}\delta(x)} \sum_{j=0}^\infty A_j(x) \tau^{-\frac j2}\ \mbox{ with some coefficients } \{ A_j(x) \}, 
$$ 
we introduce two functions $f_{\pm} = f_{\pm}(x,\lambda)$ defined for $(x,\lambda) \in \overline{\mathcal N_0} \times (0,+\infty)$ by
$$
f_\pm(x,\lambda) = e^{-\frac{\sqrt{\lambda}}{\sqrt{\sigma_s}} \delta(x)}\left[A_0(x) + \frac {\sqrt{\sigma_s}}{\sqrt{\lambda}}A_\pm(x)\right],
$$
where 
\begin{eqnarray*}
&&A_0(x) = \left\{\prod\limits_{j=1}^{N-1}\Bigl[1-\kappa_j(y(x))\delta(x)\Bigr]\right\}^{-\frac 12},
\\
&&A_\pm(x) = \int_0^{\delta(x)}\left[\frac 12\,\Delta A_0(x(\tau)) \pm 1\right]\exp\left(-\frac 12\,\int_\tau^{\delta(x)} \Delta \delta(x(\tau')) d\tau'\right)d\tau,
\end{eqnarray*}
with $x(\tau) = y(x) - \tau\,\nu(y(x))$ for $0<\tau<\de(x)$.  It is shown in \cite[Lemmas 14.16 and 14.17, p. 355]{GT1983} that 
$$
|\nabla \delta(x)| = 1\ \mbox{ and }\ \Delta \delta(x) = - \sum_{j=1}^{N-1}\frac {\kappa_j(y(x))}{1-\kappa_j(y(x)) \delta(x)}.
$$
With these in hand, by straightforward computations we obtain that
\begin{equation}
\label{gradient of two functions}
\nabla\delta\cdot\nabla A_0 = -\frac 12(\Delta\delta)A_0, \quad \nabla\delta\cdot\nabla A_\pm = -\frac 12(\Delta\delta)A_\pm + \frac 12 \Delta A_0 \pm 1 \ \mbox{ in }  \ \overline{\mathcal N_0}, 
\end{equation}
\begin{equation}
\label{for super and subsolutions}
\sigma_s\Delta f_\pm - \lambda f_\pm = \sigma_s e^{-\frac{\sqrt{\lambda}}{\sqrt{\sigma_s}} \delta(x)}\left(\mp 2 + \frac {\sqrt{\sigma_s}}{\sqrt{\lambda}}\Delta A_\pm\right)  \ \mbox{ in }  \ \overline{\mathcal N_0},
\end{equation}
and
\begin{equation}
\label{inner boundary Dirichlet}
A_0 = 1,\ A_\pm = 0,  \quad f_\pm =1  \ \mbox{ on } \ \Gamma, 
\end{equation}
for every $\la>0$.
\par
Since $\Gamma$ is of class $C^6$ and compact, we observe from \eqref{c6 regularity}--\eqref{upper bound of the curvatures on Gamma} that 
$$
|\Delta A_\pm| \le c_0\ \mbox{ in } \overline{\mathcal N_0},
$$
for some positive constant $c_0$. Therefore, it follows from \eqref{for super and subsolutions}, \eqref{exponential decay for elliptic eq} and the definition of $f_\pm$ 
that there exist two positive constants $\lambda_0$ and $\eta$ such that 
\begin{eqnarray}
&&\sigma_s\Delta f_+ - \lambda f_+ < 0 < \sigma_s\Delta f_- - \lambda f_-\ \mbox{ in } \ \overline{\mathcal N_0},\label{key differential inequalities}
\\
&& \max\{ |f_+|, |f_-|, w \} \le e^{-\eta\sqrt{\lambda}}\ \mbox{ on } \ \partial \mathcal N_0 \cap\Omega,\label{key inner boundary decay estimates}
\end{eqnarray}
for every  $\lambda \ge \lambda_0$.

\subsection{Construction of barriers for $w(x,\la)$}\label{4.3}

Let $\psi = \psi(x)$ be the unique solution of the Dirichlet problem:
$$
\Delta \psi = 0\ \mbox{ in } \ \mathcal N_0,\quad \psi = 0\ \mbox{ on } \ \Gamma, \quad \psi(x) = 2\ \mbox{ on } \ \partial \mathcal N_0 \cap\Omega.
$$
For every $(x,\lambda) \in \overline{\mathcal N_0} \times (0,+\infty)$,  we define the two functions $w_{\pm} = w_{\pm}(x,\lambda)$  by
$$
w_\pm(x,\lambda) = f_\pm(x,\lambda) \pm \psi(x)e^{-\eta\sqrt{\lambda}}.
$$
Then, in view of \eqref{elliptic pde for lambda}, \eqref{boundary condition for lambda}, \eqref{inner boundary Dirichlet}, \eqref{key differential inequalities} and \eqref{key inner boundary decay estimates}, we notice that 
\begin{eqnarray}
& \sigma_s\Delta w_+ - \lambda w_+ < 0 = \sigma_s\Delta w - \lambda w < \sigma_s\Delta w_- - \lambda w_-\ &\mbox{ in } \mathcal N_0,\nonumber
\\
& w_+ = w =w_- = 1\ &\mbox{ on } \Gamma,\label{key equality for gradient estimates}
\\
& w_- < w < w_+\ &\mbox{ on } \partial \mathcal N_0 \cap\Omega,\nonumber
\end{eqnarray}
for every  $\lambda \ge \lambda_0$,  and hence we get that
$$
w_- < w < w_+\ \mbox{ in } \ \mathcal N_0,
$$
for every  $\lambda \ge \lambda_0$, by the strong comparison principle.
Hence, combining these inequalities with \eqref{key equality for gradient estimates} and \eqref{constant Neumann serrin type} yields that 
\begin{equation}
\label{key inequalities on Gamma}
\si_s\,\pa_\nu w_+\le d_0(\la)\le \si_s\,\pa_\nu w_- \ \mbox{ on } \ \Gamma,
\end{equation}
for every  $\lambda \ge \lambda_0$.
Thus, by recalling the definition of $w_\pm$, an easy computation with  \eqref{inner boundary Dirichlet} and \eqref{gradient of two functions}  at hand gives that 
\begin{multline}
\label{bounds-for-distance-and-curvatures}
\frac 12\,\Delta\delta - \frac {\sqrt{\sigma_s}}{\sqrt{\lambda}}\left(\frac 12\Delta A_0+1\right) + (\pa_\nu \psi)\,  e^{-\eta\sqrt{\lambda}} \le  \frac{d_0(\lambda)}{\sigma_s}-
\frac {\sqrt{\lambda}}{\sqrt{\sigma_s}} \le
\\
\frac 12\,\Delta\delta - \frac {\sqrt{\sigma_s}}{\sqrt{\lambda}}\left(\frac 12\Delta A_0-1\right) -  (\pa_\nu \psi)\, e^{-\eta\sqrt{\lambda}} \ \mbox{ on } \ \Gamma,
\end{multline}
for every  $\lambda \ge \lambda_0$.

\subsection{Conclusion of the proof of Theorem \ref{th:constant flow serrin}}
By observing that the expression in the middle of the chain of inequalities \eqref{bounds-for-distance-and-curvatures} is independent of the choice of the point $x \in \Gamma$ and both sides of \eqref{bounds-for-distance-and-curvatures} have the common limit $\frac 12 \Delta \delta(x)$ as $\lambda  \to + \infty$,  we see that $\Delta\delta$ must be constant on $\Gamma$. 
Since $\Delta\delta = -\sum\limits_{j=1}^{N-1}\kappa_j$ on $\Gamma$, Aleksandrov's sphere theorem \cite[p. 412]{Alek1958vestnik} implies that
$\Gamma$ must be a sphere. 
\par
Once we know that $\Gamma$ is a sphere, by \eqref{elliptic pde for lambda}, \eqref{boundary condition for lambda} and \eqref{constant Neumann serrin type},  with the aid of the uniqueness of the solution of the Cauchy problem for elliptic equations, we see that $v$ is radially symmetric with respect to the center of $\Gamma$ in $\Omega \setminus \overline{D}$  for every $\lambda>0$, since $\Omega \setminus \overline{D}$ is connected. In particular, \eqref{relationship with the auxiliary function in section 3} yields that the function $U$ defined in Subsection \ref{subsec:initial-boundary} is radially symmetric in $\overline{\Omega} \setminus D$. Therefore, since $U = 0$ on $\partial\Omega$ and $\Omega \setminus \overline{D}$ is connected, the radial symmetry of $U$ implies that $\Omega$ must be either a ball or a spherical shell. The rest of the proof runs as explained in Subsection \ref{subsec:initial-boundary}.


\subsection{Cauchy problem: a stationary isothermic surface at the boundary}
\label{subsection4.5}

The techniques just established help us to carry out the proof of Theorem \ref{th:stationary isothermic three-phase}.
\par
Let $u$ be the solution of  problem  \eqref{heat Cauchy},  and let 
$\Gamma$ be a connected component of $\partial\Omega$.
Similarly to Subsection \ref{subsec:laplace}, we define the function $w = w(x,\lambda)$ by
$$
w(x,\lambda) = \lambda \int_0^\infty e^{-\lambda t}u(x,t)\ dt\ \mbox{ for } (x,\lambda) \in \mathbb R^N \times (0,+\infty).
$$
Item (1) of Lemma \ref{le:initial behavior and decay at infinity} ensures that
$0 < w < 1$ in $\rn\times (0,+\infty)$.

In view of the assumption \eqref{stationary isothermic surface partially}, we  set
$$
\widetilde{a}(\lambda) =  \lambda \int_0^\infty e^{-\lambda t}a(t)\,dt\ \mbox{ for } \la\in (0,+\infty).
$$
Then,  since $0 < a(t) < 1$ for every $t > 0$, it follows from Proposition \ref{prop:the initial limits on the interface} that
\begin{equation}
\label{properties of tild a}
0 < \widetilde{a}(\lambda) < 1\ \mbox{ for every } \lambda > 0\ \mbox{ and }\ \widetilde{a}(\lambda) \to \frac {\sqrt{\sigma_m}}{\sqrt{\sigma_s}+\sqrt{\sigma_m}}\ \mbox{ as }\ \lambda \to +\infty.
\end{equation}
Since $w= \widetilde{a}$ on $\Gamma \times (0,+\infty)$,  barriers for $w$ in the inner neighborhood $\mathcal N_0$ of $\Gamma$ given by \eqref{inner tubular neighborhood of Omega} can be constructed by modifying those in Subsections \ref{4.2} and \ref{4.3}. To be precise,  we set 
$$
w_\pm(x,\lambda) = \widetilde{a}(\lambda) f_\pm(x,\lambda) \pm \psi(x)e^{-\eta\sqrt{\lambda}}
\ \mbox{ for } \ (x,\lambda) \in \overline{\mathcal N_0} \times (0,+\infty),
$$
where $f_\pm, \psi, \eta$ are given in Subsections \ref{4.2} and \ref{4.3}. Then, in view of \eqref{elliptic pde for lambda},  \eqref{inner boundary Dirichlet}, \eqref{key differential inequalities} and \eqref{key inner boundary decay estimates}, for every  $\lambda \ge \lambda_0$ we verify that 
\begin{eqnarray*}
& \sigma_s\Delta w_+ - \lambda w_+ < 0 = \sigma_s\Delta w - \lambda w < \sigma_s\Delta w_- - \lambda w_-\ &\mbox{ in } \mathcal N_0,\nonumber
\\
& w_+ = w =w_- =  \widetilde{a}(\lambda)\ &\mbox{ on } \Gamma,\label{key equality for gradient estimates 2}
\\
& w_- < w < w_+\ &\mbox{ on } \partial \mathcal N_0 \cap\Omega.\nonumber
\end{eqnarray*}
These inequalities imply that
$$
w_- < w < w_+\ \mbox{ in } \ \mathcal N_0,
$$
by the strong comparison principle, and hence
\begin{equation}
\label{key inequalities on Gamma 2}
\pa_\nu w_+\le (\pa_\nu w)_- \le \pa_\nu w_- \ \mbox{ on } \ \Gamma,
\end{equation}
for every  $\lambda \ge \lambda_0$, where by $(\pa_\nu w)_-$ we mean the normal derivative of $w$ on $\Ga$ from inside of $\Omega$.
Thus, by recalling the definition of $w_\pm$, a routine computation with  \eqref{inner boundary Dirichlet} and \eqref{gradient of two functions}  at hand gives that 
\begin{multline}
\label{bounds-for-distance-and-curvatures 2}
\frac {\sigma_s\,\widetilde{a}(\lambda)}2\,\Delta\delta - \widetilde{a}(\lambda)\frac {\sigma_s\sqrt{\sigma_s}}{\sqrt{\lambda}}\left(\frac 12\,\Delta A_0+1\right) +\sigma_s\, (\pa_\nu \psi)\,  e^{-\eta\sqrt{\lambda}} \le   \sigma_s\,(\pa_\nu w)_- -
\widetilde{a}(\lambda) \sqrt{\sigma_s}\sqrt{\lambda} \le
\\
\frac {\sigma_s\widetilde{a}(\lambda)}2\,\Delta\delta -\widetilde{a}(\lambda) \frac {\sigma_s\sqrt{\sigma_s}}{\sqrt{\lambda}}\left(\frac 12\,\Delta A_0-1\right) -  \sigma_s\,(\pa_\nu \psi)\, e^{-\eta\sqrt{\lambda}} \ \mbox{ on } \ \Gamma,
\end{multline}
for every  $\lambda \ge \lambda_0$. Since $\Delta\delta = -\sum\limits_{j=1}^{N-1}\kappa_j$ on $\Gamma$, from   \eqref{bounds-for-distance-and-curvatures 2} and the second formula in \eqref{properties of tild a},  after some simple manipulation we obtain  that
\begin{equation}
\label{estimate of mean curvature from inside}
-\frac {\sigma_s\tilde{a}(\lambda)}2\sum\limits_{j=1}^{N-1}\kappa_j = \sigma_s\,(\pa_\nu w)_- -
\tilde{a}(\lambda) \sqrt{\sigma_s}\sqrt{\lambda} +O\bigl(1/\sqrt{\la}\bigr) \ \mbox{ as } \lambda \to +\infty.
\end{equation}
\par
Next, we consider  the positive function
$1-w$ in the outer neighborhood of $\Gamma$ defined by
$
\widetilde{\mathcal N}_0 = \{ x \in \mathbb R^N \setminus \overline{\Omega}\ :\ 0< \delta(x) < \delta_0 \}.
$
By similar arguments as above, since $1-w = 1-\widetilde{a}(\la)$ on $\Gamma \times (0,+\infty)$, we can construct barriers for $1-w$ on $\widetilde{\mathcal N}_0$, with the aid of the second formula of (2) of Lemma \ref{le:initial behavior and decay at infinity} and by replacing $\sigma_s, \widetilde{a}(\lambda)$ with $\sigma_m, 1-\widetilde{a}(\lambda)$. Thus, by proceeding similarly, we infer that
\begin{equation}
\label{estimate of mean curvature from outside}
+\frac {\sigma_m\, [1-\widetilde{a}(\lambda)]}2\sum\limits_{j=1}^{N-1}\kappa_j = \sigma_m\,(\pa_\nu w)_+ -
[1-\widetilde{a}(\lambda)]\, \sqrt{\sigma_m}\sqrt{\lambda} +O(1/\sqrt{\la}) \ \mbox{ as } \lambda \to +\infty,
\end{equation}
where $(\pa_\nu w)_+$ denotes the normal derivative from outside of $\Omega$ and we have taken into account both the sign of the mean curvature and the normal direction to $\Gamma$. 
\par
Now, with the aid of the transmission condition $\sigma_s\,(\pa_\nu w)_-=\sigma_m\,(\pa_\nu w)_+$ on $\Gamma$, by subtracting  \eqref{estimate of mean curvature from inside} from \eqref{estimate of mean curvature from outside}, we conclude from \eqref{properties of tild a} that 
$$
\sum\limits_{j=1}^{N-1}\kappa_j = 2\,
\frac{\widetilde{a}(\lambda)\, \sqrt{\sigma_s}-
[1-\widetilde{a}(\lambda)]\, \sqrt{\sigma_m}}{\sigma_m\,[1-\widetilde{a}(\lambda)] + \sigma_s\,\widetilde{a}(\lambda)}\,\sqrt{\lambda} +O(1/\sqrt{\la}) \ \mbox{ as } \lambda \to +\infty.
$$
\par
Since the first term at the right-hand side is independent of the choice of the point $x\in\Ga$, this formula implies that the first term has a finite limit as $\la\to \infty$ which is independent of $x\in\Ga$. 
Therefore, the mean curvature of $\Gamma$ must be constant, that is, $\Gamma$ must be a sphere.  
\par
Once we know that $\Gamma$ is a sphere, combining \eqref{stationary isothermic surface partially} with the initial  condition in \eqref{heat Cauchy} yields that, for every $t>0$, $u$ is radially symmetric in $x$ with respect to the center of $\Gamma$ in the connected component of $\mathbb R^N \setminus \overline{\Omega}$ with boundary $\Gamma$. Hence, by the transmission conditions on $\partial\Omega\ ( \supset \Gamma ) $, the function $w$ satisfies the overdetermined boundary conditions on $\Gamma$ for every $\lambda > 0$. Then, since $\sigma_s \Delta w - \lambda w = 0$ in $\Omega\setminus\overline{D}$ and $\Omega \setminus \overline{D}$ is connected, with the aid of the uniqueness of the solution of the Cauchy problem for elliptic equations, we see that $w$ is radially symmetric with respect to the center of $\Gamma$ in $\overline{\Omega}\setminus D$ for every $\lambda > 0$. This means that $u$ is radially symmetric in $x$ with respect to the center of $\Gamma$ in $\left(\overline{\Omega}\setminus D\right)\times (0,+\infty)$. 
 Moreover, as in the proof of Theorem \ref{th:constant flow} for problem  \eqref{heat Cauchy},  in view of the initial  condition in \eqref{heat Cauchy} and Proposition \ref{prop:the initial limits on the interface}, we can prove that $\Omega$ is radially symmetric and hence $u$ is radially symmetric in $x$  with respect to the center of $\Gamma$  on $\mathbb R^N\setminus D$ for every $t>0$.  
\par
The rest of the proof runs as that of Theorem \ref{th:constant flow}  for problem \eqref{heat Cauchy} in Subsection \ref{subsection 3.3}. \qed

\setcounter{equation}{0}

\section{The Cauchy problem when $\sigma_s=\sigma_m$}
\label{section5}

Here, we present the proof of Theorem \ref{th:stationary isothermic on the boundary for cauchy}, that is
$u$ is the solution of problem \eqref{heat Cauchy} with $\sigma_s=\sigma_m$.  
For a connected component $\Gamma$ of $\partial\Omega$, set the positive constant 
\begin{equation}
\label{distance to the inclusion D}
\rho_0 = \mbox{ dist}(\Gamma, \overline{D}).
\end{equation}

\subsection{Proof of proposition (a)}  Let $p, q \in \Gamma$ be two distinct points and introduce a function $v = v(x,t)$ by
$$
v(x,t) = u(x+p, t) - u(x + q, t)\ \mbox{ for } (x,t)\in B_{\rho_0}(0) \times (0,+\infty).
$$
Then, since $\sigma = \sigma_s$ in $\mathbb R^N\setminus\overline{D}$, we observe from \eqref{stationary isothermic surface partially} that
\begin{equation*}
v_t =\sigma_s \Delta v\ \mbox{ in }\ B_{\rho_0}(0)\times (0,+\infty)\ \mbox{ and }\ v(0,t) = 0\ \mbox{ for every } t > 0.
\end{equation*}
Therefore we can use a balance law (see \cite[Theorem 2.1, pp. 934--935]{MSannals2002} or \cite[Theorem 4, p. 704]{MSmathz1999}) to obtain that
$$
\int\limits_{B_r(0)}\!\! v(x,t)\ dx = 0\ \mbox{ for every }\  (r,t) \in (0, \rho_0)\times(0,+\infty).
$$
Thus, in view of the initial  condition of problem  \eqref{heat Cauchy}, letting $t \to +0$ yields that
\begin{equation}
\label{uniformly dense}
|\Omega^c\cap B_r(p)| = |\Omega^c\cap B_r(q)|\ \mbox{ for every } r \in (0,\rho_0),
\end{equation}
where the bars indicate the Lebesgue measure of the relevant sets. This means that ${\overline{\Omega}}^c$ is uniformly dense in $\Gamma$ in the sense of \cite[(1.4), p. 4822]{MPS2006tams}. 
\par
Therefore, \cite[Theorem 1.2, p. 4823]{MPS2006tams} applies and we see that $\Gamma$ must have constant mean curvature. Again, Aleksandrov's sphere theorem implies that
$\Gamma$ is a sphere. By combining \eqref{stationary isothermic surface partially} and the initial  condition in \eqref{heat Cauchy}  with the real analyticity in $x$ of $u$ over $\mathbb R^N \setminus\overline{D}$,  we see that $u$ is radially symmetric  in $x$ with respect to the center of $\Gamma$ on $\left(\mathbb R^N\setminus D\right) \times (0,+\infty)$. Here we used the fact that $\mathbb R^N\setminus\overline{D}$ is connected. Then, the rest of the proof runs as in the proof of Theorem \ref{th:constant flow} for problem \eqref{heat Cauchy} in Subsection \ref{subsection 3.3}. 

\subsection{Proof of proposition (b)} With the aid of a balance law (see \cite[Theorem 2.1, pp. 934--935]{MSannals2002} or \cite[Theorem 4, p. 704]{MSmathz1999}) and the assumption \eqref{constant flow surface partially}, by the same argument as in the proof of Lemma \ref{le: constant weingarten curvature}, we obtain the same equality as \eqref{balance law special}:
\begin{equation*}
\label{constant flow with balance law 2}
\nu(p)\cdot\!\!\!\!\int\limits_{B_{r}(p)}\!\!\!\!u(x,t)(x-p)\, dx= \nu(q)\cdot\!\!\!\! \int\limits_{B_{r}(q)}\!\!\!\!u(x,t)(x-q)\, dx\ \mbox{ for }\ (r,t) \in (0, \rho_0)\times(0,+\infty),
\end{equation*}
where $p, q \in \Gamma$ and $\nu$ is the outward unit normal to $\partial\Omega$. Then, in view of the initial  condition in \eqref{heat Cauchy}, letting $t \to +0$ yields that for every $p, q \in \Gamma$
\begin{equation}
\label{uniformly dense like}
\nu(p)\cdot\!\!\!\!\!\!\!\!\int\limits_{\Omega^c\cap B_{r}(p)}\!\!\!\!\!\!(x-p)\, dx= \nu(q)\cdot\!\!\!\!\!\!\!\!\int\limits_{\Omega^c\cap B_{r}(q)}\!\!\!\!\!\!(x-q)\, dx \ \mbox{ for }\ r \in (0, \rho_0).
\end{equation}
\par
The use of the techniques established in \cite{MPS2006tams} gives the asymptotic expansion 
\begin{equation}
\label{asymptotic expansion 6}
\nu(p)\cdot\!\!\!\!\!\!\!\!\int\limits_{\Omega^c\cap B_{r}(p)}\!\!\!\!\!\!(x-p)\, dx=\frac{\omega_{N-1}}{N^2-1}\,r^{N+1}
\left[ 1 - \frac{C(p)}{8(N+3)} r^2 + o(r^2)\right] \ \mbox{ as } \ r\to 0,
\end{equation}
where $\omega_{N-1}$ is the volume of the unit sphere $\SS^{N-2}\subset \RR^{N-1}$ and 
\begin{equation}
\label{new symmetric function of principal curvatures}
C(p)= \begin{cases} 3\sum_{i=1}^{N-1} \kappa_i^2(p)+ 2 \sum\limits_{i<j} \kappa_i(p)\kappa_j(p)\ &\mbox{ if }\  N \ge 3,
\\
3\kappa_1^2(p) \ &\mbox{ if }\  N = 2.
\end{cases}
\end{equation}
Indeed, by introducing the spherical coordinates  as in \cite[(5.1), p. 4835]{MPS2006tams} where we choose the origin as the point $p \in \Gamma$ and $\nu$ as the outward unit normal vector to $\partial\Omega$, \cite[(5.5), p. 4835]{MPS2006tams} is replaced with
\begin{eqnarray}
\nu(p)\cdot\!\!\!\!\!\!\!\!\int\limits_{\Omega^c\cap B_{r}(p)}\!\!\!\!\!\!(x-p)\, dx&=& \int\limits_{\mathbb S^{N-2}}\!\!\int\limits_0^r\!\rho^N\!\!\!\!\int\limits_{\theta(\rho,v)}^{\pi/2}\!\!\!\sin\phi\cos^{N-2}\!\phi\ d \phi d\rho dS_v \nonumber
\\
&=& \frac 1{N-1}\int\limits_0^r\rho^N\int\limits_{\mathbb S^{N-2}}\!\!\cos^{N-1}\!\theta(\rho,v)\  dS_v d\rho, \label{key identity}
\end{eqnarray}
where $dS_v$ denotes the surface element on $\mathbb S^{N-2}$. 
Since $\partial\Omega$ is of class $C^2$,   \cite[(5.4), p. 4835]{MPS2006tams} is replaced with
$$
 \theta(\rho, v) = \theta_1(v) \rho + o(\rho)\ \mbox{  as } \rho \to 0.
$$
Thus,  using the formula 
$$
\cos^{N-1}\!\theta = 1 -\frac {N-1}{2} \theta^2 + O(\theta^4)\ \mbox{ as } \theta \to 0,
$$
 yields that 
\begin{equation}
\label{key pre-asymptotics 1}
\cos^{N-1}\!\theta(\rho, v) = 1 -\frac {N-1}{2} \theta_1(v)^2 \rho^2  + o(\rho^2)\  \mbox{ as } \rho \to 0.
\end{equation}
In the beginning of \cite[p. 4837]{MPS2006tams} we know that
$$
\theta_1(v) = P_2(v) = -\frac 12\sum_{j=1}^{N-1}\kappa_j(p) v_j^2 \ \mbox{ for } \ v \in \mathbb S^{N-2} \left(\subset \mathbb R^{N-1}\right),
$$
since  \cite[(5.6), p. 4836]{MPS2006tams} is replaced with 
 $$
 \varphi(y) = P_2(y) +o(|y|^2)\ \mbox{ as } y \to 0 \mbox{ in } \mathbb R^{N-1}.
 $$
With \cite[Lemma 5.4, p. 4837]{MPS2006tams} in hand, we calculate that for $N \ge 3$
 \begin{eqnarray}
 \int\limits_{\mathbb S^{N-2}} \theta_1(v)^2 \ dS_v &=& \frac 14  \int\limits_{\mathbb S^{N-2}} \left(\sum_{j=1}^{N-1}\kappa_j(p)v_j^2\right)^2 dS_v\nonumber
 \\
& =& \frac 14\left\{ \sum_{j=1}^{N-1} \kappa^2_j(p)\!\!\!\int\limits_{\mathbb S^{N-2}}\!\!\!v_j^4\ dS_v+ 2 \sum_{i<j} \kappa_i(p)\kappa_j(p)\!\!\!\int\limits_{\mathbb S^{N-2}}\!\!\! v_i^2v_j^2\ dS_v\right\}\nonumber
\\
&=& \frac {\omega_{N-1}}{4(N^2-1)} \left\{ 3\sum_{j=1}^{N-1} \kappa^2_j(p) +   2\sum_{i<j} \kappa_i(p)\kappa_j(p)\right\}, \label{key pre-asymptotics 2}
 \end{eqnarray}
 and for $N=2$ 
 \begin{equation}
 \label{key pre-asymptotics 3}
  \int\limits_{\mathbb S^{N-2}} \theta_1(v)^2\ dS_v = \frac 12 \kappa^2_1(p). 
  \end{equation}
 Therefore it follows from  \eqref{key identity}, \eqref{key pre-asymptotics 1}, \eqref{key pre-asymptotics 2} and \eqref{key pre-asymptotics 3} that \eqref{asymptotic expansion 6} holds true.
Thus, by combining \eqref{asymptotic expansion 6} with \eqref{uniformly dense like}, we reach the conclusion that  $C(p)$ must be constant on $\Gamma$.
\par
If $N=2$, this directly implies that $\Gamma$ is a (closed) curve of constant curvature, hence a circle.
If $N \ge 3$, the equation that $C(p) $ is a constant on $\Gamma$  means that $\Gamma$ is an elliptic Weingarten-type surface considered by Aleksandrov \cite[p. 412]{Alek1958vestnik}, where the ellipticity follows from the strict convexity $\min\limits_{1\le j \le N-1}\kappa_j > 0$.  Thus Aleksandrov's sphere theorem implies that $\Gamma$ must be a sphere.
Then, we conclude by the same reasoning as in the proof of Theorem \ref{th:constant flow} for problem \eqref{heat Cauchy} in Subsection \ref{subsection 3.3}.

\section*{Acknowledgements}
The first and third authors were partially supported by the Grants-in-Aid
for Scientific Research (B) ($\sharp$ 26287020,  $\sharp$ 18H01126), Challenging Exploratory Research ($\sharp$ 16K13768) and JSPS Fellows ($\sharp$ 18J11430) of
Japan Society for the Promotion of Science. The second author was partially supported by an iFUND-Azione 2-2016 grant of the Universit\`a di Firenze.
The authors are grateful to the anonymous reviewers for their many valuable comments and remarks to improve clarity in many  points.


\end{document}